\documentclass[11pt,a4paper]{amsart}
\usepackage{amssymb,latexsym}
\usepackage[utf8]{inputenc}
\usepackage[T1]{fontenc}
\usepackage[francais,english]{babel}
\usepackage{enumerate}
\usepackage{enumitem}
\usepackage[normalem]{ulem} 
\usepackage{graphicx,amssymb,amsfonts,epsfig,amsthm,a4,amsmath,url,amscd}
\usepackage{color}
\usepackage{csquotes}

\title{Densely related groups}
\author{Yves Cornulier}
\address{Laboratoire de Math\'{e}matiques, B\^{a}timent 425, Universit\'{e} Paris-Sud 11, 91405 Orsay, France}
\email{yves.cornulier@math.u-psud.fr}
\author{Adrien Le Boudec}
\thanks{The authors were supported by ANR-14-CE25-0004 GAMME}
\thanks{ALB is a F.R.S.-FNRS Postdoctoral Researcher.}
\address{UCLouvain, IRMP,	Chemin du Cyclotron 2, 1348 Louvain-la-Neuve, Belgium}
\email{adrien.leboudec@uclouvain.be}

\date{August 7, 2017}

\makeatletter
\@namedef{subjclassname@2010}{%
  \textup{2010} Mathematics Subject Classification}
\makeatother


\newtheorem{thm}{Theorem}[section]
\newtheorem{cor}[thm]{Corollary}
\newtheorem{lem}[thm]{Lemma}

\newtheorem{prop}[thm]{Proposition}
\theoremstyle{definition}
\newtheorem{defn}[thm]{Definition}
\newtheorem{definition}[thm]{Definition}

\newtheorem{que}[thm]{Question}
\newtheorem{conj}[thm]{Conjecture}
\theoremstyle{remark}
\newtheorem{rem}[thm]{Remark}
\newtheorem{ex}[thm]{Example}

\numberwithin{equation}{section}

\newcommand{\Z}{\mathbf{Z}}

\newcommand{\N}{\mathbf{N}}
\newcommand{\R}{\mathbf{R}}

\newcommand{\K}{\mathbf{K}}

\newcommand{\RRS}{\mathcal{R}_S(G)}
\newcommand{\RR}{\mathcal{R}(G)}

\begin{document}

\maketitle

\begin{abstract}

We study the class of densely related groups. These are finitely generated (or more generally, compactly generated locally compact) groups satisfying a strong negation of being finitely presented, in the sense that new relations appear at all scales. Here, new relations means relations that do not follow from relations of smaller size. Being densely related is a quasi-isometry invariant among finitely generated groups. 

We check that a densely related group has none of its asymptotic cones simply connected. In particular a lacunary hyperbolic group cannot be densely related. 

We prove that the Grigorchuk group is densely related. We also show that a finitely generated group that is (infinite locally finite)-by-cyclic and which satisfies a law must be densely related. Given a class $\mathcal{C}$ of finitely generated groups, we consider the following dichotomy: every group in $\mathcal{C}$ is either finitely presented or densely related. We show that this holds within the class of nilpotent-by-cyclic groups and the class of metabelian groups. In contrast, this dichotomy is no longer true for the class of $3$-step solvable groups. 
\end{abstract}








\section{Introduction}

If $G$ is a group and $S$ a generating subset, a \textbf{relation} in $G$ is by definition an element of the kernel of the natural map $F_S \twoheadrightarrow G$. A relation $w \in F_S$ is said to be generated by a subset $R \subset F_S$ if $w$ belongs to the normal subgroup of $F_S$ generated by $R$. Equivalently, $w$ can be written as a product of conjugates of elements of $R^{\pm 1}$. To every pair $(G,S)$, where $G$ is a group and $S$ a generating subset, we study the {\bf relation range} $\RRS$, which is the set of lengths of relations in $G$ that are not generated by relations of smaller length. Here the length of an element of $F_S$ refers to the word length associated to $S$. The set $\RRS$ is finite precisely when the kernel of $F_S \twoheadrightarrow G$ is normally generated by elements of bounded length. When $S$ is finite, this exactly means that the group $G$ is finitely presented. The relation range was introduced by Bowditch in \cite{Bowditch-RR}, but has not been explicitly considered since then, except in the small cancelation case \cite{Thomas-Borel}.

The study of the relation range takes part in a program to find ``measures" of the failure of finite presentability for finitely generated groups. Several approaches, of independent interest, have been carried out so far:
\begin{itemize}
\item In \cite{BCGS}, several strengthenings of infinite presentability have been investigated, in terms of the study of the poset of normal subgroups contained in the kernel of epimorphisms from finitely presented groups to a given finitely generated group. However, these properties are not known to be quasi-isometry invariant, or are known not to be (see Remark \ref{abelscon}). 
\item In \cite{CT-Abels}, to one finitely generated group $G$, one associates the set of nonprincipal ultrafilters $\omega$ such that the asymptotic cone $\mathrm{Cone}^{\omega}(G,(1/n))$ is simply connected. This is obviously a quasi-isometry invariant.
\end{itemize}

The study of the relation range is in spirit closer to the latter one, but it is not based on asymptotic cones (although we make a connection at some point) and therefore more of combinatorial flavor. The relation range has a distinct behavior than the aforementioned invariant from \cite{CT-Abels}, and may detect properties not seen by this invariant (see Remark \ref{rem-diff-CT}).

We stress out that the relation range of a finitely generated group is not defined in terms of a presentation of the group. Moreover in general there is no connection between the relation range and the set of lengths of relators in a given presentation of the group, even in the case of a minimal presentation (see Remark \ref{rem-min-RR}).

We should mention that there exist other notions turning out to be strengthenings of infinite presentability, such as the negation of being of homological type $\mathrm{FP}_2$ over a given commutative ring. This is a quasi-isometry invariant;  Alonso's proof from \cite{Alo} in the case of the ring $\Z$ carrying over the general case.

\begin{defn}
A finitely generated group $G$ is {\bf densely related} if there exists $c>0$ such that for all $n\ge 1$, we have $\RRS\cap [n,cn]\neq\emptyset$. (This does not depend on the finite generating subset $S$, see Corollary \ref{cor-intro-qi-inv}.)  
\end{defn}

These are groups in which new relations (that is, not generated by relations of smaller length) appear at all scales, in a sense precisely defined in the sequel. A typical example of a densely related group is the wreath product $\Gamma = \Z \wr \Z$, which admits the presentation $\left\langle t,x \, \mid \, [t^nxt^{-n},x]=1 \, \, \text{for all} \, \, n \geq 1\right\rangle$. By a result of Baumslag \cite{Bau-61}, this presentation is minimal. The relation $[t^nxt^{-n},x]$ is not generated by relations of smaller length, and therefore $\mathcal{R}_{\{t,x\}}(\Gamma)$ contains the integer $4n+4$ for every $n \geq 1$.

\medskip

Recall that Gromov proved that if a finitely generated group $G$ has all its asymptotic cones simply connected, then $G$ is finitely presented and has a polynomially bounded Dehn function \cite[\S 5.F]{Gro-asy}. In turns out that simple connectedness of \textit{one} asymptotic cone of $G$ already has consequences, namely that $G$ cannot be densely related (see Proposition \ref{prop-1-cone-sc} for a more precise result):

\begin{thm}[Corollary \ref{cor-full-pres-cones}] \label{ascone_intro}
Densely related finitely generated groups have no simply connected asymptotic cone.
\end{thm}

\medskip

Any (non-trivial) wreath product is an example of a densely related group (Proposition \ref{prop-wreath-dr}). The following construction provides other kind of examples: Let $G$ be a non-Hopfian group, and let $\varphi: G \rightarrow G$ be a surjective endomorphism of $G$ with non-trivial kernel. We denote by $K_{\varphi}$ the increasing union of the normal subgroups $\ker (\varphi^n)$, $n \geq 1$, and by $G_{\varphi} = G / K_{\varphi}$ the associated quotient. 

\begin{prop}[Corollary \ref{cor-lim-hopf-fully}] \label{prop-intro-hopf-densely}
Let $G$ be a compactly generated locally compact group. Assume that $\varphi$ is a continuous surjective and non-injective endomorphism of $G$ such that $K_{\varphi}$ is a closed subgroup of $G$. Then the group $G_{\varphi} = G / K_{\varphi}$ is densely related.
\end{prop}

Examples of groups covered by Proposition \ref{prop-intro-hopf-densely} are the semidirect products $\Z[1/pq] \rtimes \Z$, where the action is by multiplication by $p/q$, where $p,q\ge 2$ are coprime integers. See \S \ref{subsec-non-hopf} for more examples.


Proposition \ref{prop-intro-hopf-densely} is a particular case of a more general construction investigated in \S \ref{subsec-non-hopf}. Let $G$ be a group with a finite index normal subgroup $H$, which come with a finite set $\Phi = \left\{\varphi_i\right\}$ of surjective homomorphisms $\varphi_i: H \rightarrow G$. Under appropriate assumptions on $G$, $H$ and $\Phi$ (see \S \ref{subsec-non-hopf}), this naturally defines an increasing sequence of normal subgroups $K_n \lhd G$, so that the groups $G/K_n$ form a directed system. We show that the associated direct limit $G/K_\infty$ is always densely related (Theorem \ref{thm-lim-fully}). An example of group which may be obtained via this construction is the Grigorchuk group $\mathfrak{G}$ introduced in \cite{grig80}, so that we obtain the following result.

\begin{thm}
The Grigorchuk group $\mathfrak{G}$ is densely related. In particular no asymptotic cone of $\mathfrak{G}$ is simply connected.
\end{thm}

\medskip

Finitely generated groups that are not densely related are called {\bf lacunary presented}. Of course this includes finitely presented groups, but here we are especially interested by infinitely presented ones.

The idea of considering sequences of relations of sparse lengths was used by Bowditch in order to provide a continuum of pairwise non-quasi-isometric infinitely presented small cancelation groups \cite{Bowditch-RR}. It was then shown by Thomas and Velickovic that this yields the first example of a finitely generated group having non-homeomorphic asymptotic cones (more precisely, one asymptotic cone that is a real tree, and one asymptotic cone that is not simply connected) \cite{TV}. This motivated the introduction of {\it lacunary hyperbolic groups} (groups with at least one asymptotic cone a real tree) by Olshanskii, Osin and Sapir \cite{OOS}. They showed that this class of groups is actually very large. These are examples of groups that are lacunary presented, not all of which being finitely presented. However, the class of lacunary presented groups is much larger (see Remark \ref{comparisonLHLP}).

\medskip

Given $A,B \subset \R_+$, we write $A \preccurlyeq B$ if there exists $c>0$ such that for every $a \in A$, there exists an element of $B$ in the interval $[c^{-1} a,c a]$. We write $A \sim B$ when $A \preccurlyeq B$ and $B \preccurlyeq A$, and we say that $A,B$ are at finite multiplicative Hausdorff distance. This defines an equivalence relation for subsets of $\R_+$, which identifies all non-empty finite subsets. 


The following result is a particular case of Theorem \ref{thm-qi-inv}, which holds in the more general context of connected graphs. We refer to \S \ref{subsec-inv-qi} for the relevant terminology.

\begin{thm} \label{thm-intro-LS-inv}
Let $G,S$ and $G',S'$ be two pairs of groups and generating subsets. If the metric space $(G,d_S)$ is a large-scale Lipschitz retract of $(G',d_{S'})$, then $\RRS \preccurlyeq \mathcal{R}_{S'}(G')$.
\end{thm}

\begin{cor}[Bowditch] \label{cor-intro-qi-inv}
If the metric spaces $(G,d_S)$ and $(G',d_{S'})$ are quasi-isometric, then $\RRS \sim \mathcal{R}_{S'}(G')$.
\end{cor}

Corollary \ref{cor-intro-qi-inv} implies in particular that if $G$ is endowed with a (possibly discrete) locally compact topology, and if $S$ is a compact generating subset, then the set $\RRS$ does not depend, up to finite multiplicative Hausdorff distance, on the choice of $S$. 
We denote by $\RR$ the associated equivalence class, and we call $\RR$ the \textbf{relation range} of the group $G$. According to Corollary \ref{cor-intro-qi-inv}, the relation range $\RR$ in an invariant of the quasi-isometry class of $G$.

In this setting, we have the following two extreme situations:

\begin{itemize}
\item $\RR$ is finite. When $G$ is a finitely generated group, the fact that $\RR$ is finite exactly means that $G$ is finitely presented. More generally in the locally compact setting, the groups having a finite relation range are precisely the compactly presented groups. For an introduction to compactly presented groups, we refer the reader to the book \cite{Cor-dlH}.
\item $\RR \sim \N$. By definition the groups $G$ such that $\RR \sim \N$ are the densely related groups.
\end{itemize}

It is natural to ask which subsets of $\N$ are, up to multiplicative Hausdorff distance, the relation range of a finitely generated group. The easy answer is: all of them. Indeed we can start from the above presentation of $\Z\wr\Z$ and choose an arbitrary subset of relators. Namely, given any subset $I\subset\N\smallsetminus\{0\}$, defining $\Gamma_I=\left\langle t,x \, \mid \, [t^nxt^{-n},x]=1 \, \, \text{for all} \, \, n\in I\right\rangle$, the relation range of $(\Gamma_I,\{t,x\})$ is equal to $\{0\}\cup (4I+4)$, and every non-empty subset of $\N$ is at finite Hausdorff distance to a subset of this form. As a consequence, the relation range distinguishes continuum many quasi-isometry classes of groups, a fact that was established by Bowditch in \cite{Bowditch-RR} (using small cancelation rather than the previous presentation). Bowditch's strategy, combined with work of Olshanskii-Osin-Sapir \cite{OOS}, actually shows that every subset $I\subset\N$ such that $I \nsim \N$ is the relation range of a finitely generated lacunary hyperbolic group (see \S \ref{subsec-various-RR}).

Still, it remains natural to ask the same question within more restricted classes of groups. First recall that if $\mathcal{P},\mathcal{Q}$ are properties, a group is said to be $\mathcal{P}$-by-$\mathcal{Q}$ when there is a normal subgroup with $\mathcal{P}$ whose associated quotient has $\mathcal{Q}$.

\begin{prop} \label{prop-intro-all-subsets}
Every subset of $\N$ is (equivalent to) the relation range of a (3-nilpotent locally finite)-by-abelian finitely generated group.
\end{prop}

Note in particular that infinitely presented solvable groups can be lacunary presented. In contrast, non-elementary lacunary hyperbolic groups are never solvable. Indeed, more generally they cannot satisfy a law \cite[Cor.\ 6.13]{DS}. Recall that a group $G$ \textbf{satisfies a law} if there exists a non-trivial group word $w(x_1,\dots,x_k)$ in a free group such that $w(g_1,\dots,g_k)=1$ for all $g_1,\dots,g_k\in G$.

\medskip

Now given a certain class of groups, we want to study the behavior of the relation range within this class. Proposition \ref{prop-intro-all-subsets} says that the class of amenable (and even solvable) groups is not yet sufficiently restrictive, whereby the necessity to focus on smaller classes of groups. The following result provides classes of finitely generated groups for which the behavior of the relation range is completely understood.

\begin{thm} \label{thm-intro-classes}
Let $G$ be a finitely generated group. Assume that one of the following holds:
\begin{enumerate}
	\item $G$ is metabelian, or more generally center-by-metabelian;
	\item $G$ is nilpotent-by-cyclic.
\end{enumerate}
Then $G$ is either finitely presented or densely related.  
\end{thm}

In combination with Theorem \ref{ascone_intro}, this implies the following: 

\begin{cor}
Let $G$ be as in Theorem \ref{thm-intro-classes}. If $G$ admits one simply connected asymptotic cone, then $G$ is finitely presented.
\end{cor}

Recall that it follows from Bieri-Strebel theorem that a (non virtually cyclic) finitely generated group $G$ that is (locally finite)-by-cyclic is infinitely presented \cite{BS78}. By a construction due to Olshanskii, Osin and Sapir \cite[\S 3.5]{OOS}, (locally finite)-by-cyclic groups are not always densely related. Moreover, although we will not go into this direction here, the flexibility of their construction suggests that the relation range of (locally finite)-by-cyclic groups may be arbitrary. The following result shows that, under the additional assumption that the group satisfies a law, the relation range is forced to be as large as possible.

\begin{thm} \label{thm-intro-locfin-Z}
Let $G = N \rtimes \Z$ be a finitely generated group, where $N$ is an infinite locally finite subgroup. If $G$ satisfies a law, then $G$ is densely related. In particular $G$ has no simply connected asymptotic cone.
\end{thm}

The existence of a law in a finitely generated group or in some large enough subgroup was already known to have consequences at the level of asymptotic cones. For instance if $G$ satisfies a law then no asymptotic cone of $G$ can be tree-graded in the sense of Drutu and Sapir \cite[Cor.\ 6.13]{DS}. Also if $G$ has a subgroup $H$ of relative exponential growth that satisfies a law, then no asymptotic cone of $G$ can be real tree \cite[Th.\ 3.18(c)]{OOS}, \cite[Th.\ 1.4]{LB-lac-hyp}. The idea is that the existence of a law prevents ultraproducts from containing free subgroups, and this has consequences on their possible transitive isometric group actions. The approach of Theorem \ref{thm-intro-locfin-Z} is different in the sense that the existence of a law is interpreted directly at the level of the group (rather than ultraproducts) in order to produce appropriate relations.

\medskip

A common feature of the proofs of Theorem \ref{thm-intro-classes} and Theorem \ref{thm-intro-locfin-Z} is the study of the relation range of finitely generated groups with a given homomorphism onto the group $\Z$. Recall that given $\pi: G \twoheadrightarrow \Z$, we say that the action of $\Z$ \textbf{contracts into a finitely generated subgroup} of $G$ if there is a decomposition of $G$ as an ascending HNN-extension over a finitely generated group whose associated homomorphism onto $\Z$ is equal to $\pi$.


\begin{thm} \label{thm-intro-law-by-Z}
Let $G = M \rtimes \Z$ be a finitely generated group satisfying a law. Then either the action of $\Z$ contracts into a finitely generated subgroup of $M$, or the group $G$ is densely related.  
\end{thm}

\medskip

Finally we end this introduction with the following problem. While the class of finitely generated linear groups (that is, isomorphic to a subgroup of $\mathrm{GL}_n(K)$ for some $n$ and some field $K$) contains both finitely presented groups and densely related groups (e.g.\ the wreath product $\Z \wr \Z$), we do not know any other behavior of the relation range within this class of groups. 

\begin{que}\label{linear-question}
Is it true that every finitely generated linear group is either finitely presented or densely related ?
\end{que}

We point out that a positive answer to this question would imply the non-existence of infinitely presented linear lacunary hyperbolic groups, and the existence of such groups was asked by Olshanskii, Osin and Sapir in \cite{OOS}. We refer to \S\ref{subsec-cones} for details. 

\subsection*{Outline} This article is organized as follows. Section \ref{sec-rel-range} introduces the analogue of the relation range for graphs, and contains the proof of Theorem \ref{thm-intro-LS-inv}. In Section \ref{sec-rel-range-group} we establish preliminary results on the relation range and make the connection with asymptotic cones (Theorem \ref{ascone_intro}). Finally the study of the relation range for specific classes of groups is carried out in Section \ref{sec-onto-Z}, which contains the proofs of all other results stated in the introduction.

\subsection*{Acknowledgments} We are grateful to Goulnara Arzhantseva and Pierre de la Harpe for useful comments, to Romain Tessera for interesting discussions related this work, and to Henry Wilton for pointing out \cite{CG}. We also thank Mikhail Ershov for pointing out some inaccuracy in a preliminary version of this article, and the referee for useful suggestions. 

\section{Relation range of graphs} \label{sec-rel-range}

Following \cite{Bowditch-RR}, this section introduces the analogue of the relation range in the setting of connected graphs.


\subsection{Equivalence relation on subsets of $\R_+$}

Given two subsets $A,B$ of $\R_+$, we write $A \preccurlyeq B$ if there exists $c>0$ such that there is an element of $B$ in the interval $[c^{-1} a,c a]$ for every $a \in A$. The relation $A \sim B$ defined by $A \preccurlyeq B$ and $B \preccurlyeq A$ is an equivalence relation on the set of subsets of $\R_+$. Equivalently, we have $A \sim B$ if and only if $A,B$ are at finite multiplicative Hausdorff distance, i.e.\ $\log(A)$ and $\log(B)$ are at finite Hausdorff distance. The class of a subset $A$ will be denoted $\left[A\right]$. We leave as an exercise the verification that the operation $\left[A\right] \cup \left[B\right] = \left[A \cup B\right]$ is well defined, and that the relation defined by $\left[A\right] \subset \left[B\right]$ if there exist $A' \sim A$ and $B' \sim B$ such that $A' \subset B'$, is a partial order on the set of equivalence classes. 


\subsection{Definition for graphs}

Let $X$ be a graph. A path $\alpha$ of length $n \geq 0$ in $X$ is a sequence of vertices $x_0,\ldots,x_{n}$ such that $x_i$ and $x_{i+1}$ are adjacent for every $i=0,\ldots,n-1$. We say that $x_0$ and $x_{n}$ are respectively the inital point and the endpoint of $\alpha$. We say that $\alpha$ is a loop if $x_0 = x_{n}$. 

Assume now that $X$ is a connected graph, and choose a base point $x_0 \in X$. For every $n \geq 0$, we consider the subgroup $\pi_1^{(n)}(X,x_0)$ of $\pi_1(X,x_0)$ generated by loops of the form $p^{-1} \cdot \alpha \cdot p$, where $p$ is a path with initial point $x_0$, and $\alpha$ is a loop (based at the endpoint of $p$) of length at most $n$. The subgroup $\pi_1^{(n)}(X,x_0)$ is normal in $\pi_1(X,x_0)$, and clearly $\pi_1^{(n)}(X,x_0)$ is a subgroup of $\pi_1^{(m)}(X,x_0)$ for $m \geq n$. Note that $\pi_1(X,x_0)$ is the increasing union of the subgroups $\pi_1^{(n)}(X,x_0)$. We take the convention that $\pi_1^{(-1)}(X,x_0)$ is the trivial subgroup.

\begin{definition}
We denote by $\Phi(X)$ the set of integers $n \geq 0$ such that $\pi_1^{(n)}(X,x_0)$ properly contains $\pi_1^{(n-1)}(X,x_0)$. This does not depend on $x_0$.
\end{definition}

For simplicity we will use the same notation for $\Phi(X)$ and its $\sim$-class.


\begin{rem}
For every integer $k \geq 1$, the set $\Phi(X)$ is at finite multiplicative Hausdorff distance to the set of integers $n$ such that $\pi_1^{(kn)}(X,x_0)$ properly contains $\pi_1^{(n-1)}(X,x_0)$.
\end{rem}

\subsection{Quasi-isometry invariance} \label{subsec-inv-qi}

In this paragraph we study the behavior of $\Phi(X)$ under large scale Lipschitz retracts of connected graphs.

\medskip

A map $f: X\to Y$ between two metric spaces is $(C,C')$-LS-Lipschitz (where LS stands for Large-Scale) if \[ d(f(x),f(x')) \le \max(Cd(x,x'),C') \, \, \text{for all} \, \, x,x'\in X;\] and $f$ is \textbf{LS-Lipschitz} if it is $(C,C')$-LS-Lipschitz for some constants $C,C'$. The space $X$ is an \textbf{LS-Lipschitz retract} of $Y$ if there are LS-Lipschitz maps $X\to Y\to X$ whose composite map is at bounded distance to the identity map of $X$. 

\begin{thm}\label{thm-qi-inv}
Let $X,Y$ be connected graphs. Assume that $X$ is an LS-Lipschitz retract of $Y$. Then $[\Phi(X)]\subset[\Phi(Y)]$. In particular, if $X$ and $Y$ are quasi-isometric then $[\Phi(X)]=[\Phi(Y)]$. 
\end{thm}

Let us proceed to prove the theorem. It will follow from the more quantitative Lemma \ref{prop-qi-inv}.

If $X$ is a metric space and $c \geq 0$, let $R_c(X)$ be the Rips complex: this is the simplicial complex whose set of vertices is $X$ and there is a $n$-simplex between any $n+1$ points pairwise at distance $\le c$. (We only consider the Rips complex as a topological space, when endowed with the usual inductive limit topology.) 


If $f:X\to Y$ maps any two points at distance $\le c$ to points at distance $\le c'$, then it induces a map $\bar{f}:R_c(X)\to R_{c'}(Y)$, defined to be equal to $f$ on vertices and extended to be affine on simplices; then $\bar{f}$ is continuous and $f\mapsto\bar{f}$ is functorial. In particular: 

\begin{itemize}
\item if $f$ is $(C,C')$-LS-Lipschitz then for all $c\ge C'/C$, $f$ induces a map $\bar{f}:R_c(X)\to R_{cC}(Y)$;
\item if $f$ has distance $\le k$ to the identity, i.e.\ $d(f(x),x)\le k$ for all $x \in X$, then $d(f(x),f(x'))\le d(x,x')+2k$, so this defines $\bar{f}$ from $R_c(X)$ to $R_{c+2k}(X)$ for all $c \geq 0$.
\end{itemize}



\begin{lem}\label{distid}
If $f$ has distance $\le k$ to the identity then $\bar{f}:R_c(X)\to R_{c+2k}(X)$ is homotopic to $\bar{i}$, where $i$ is the identity of $X$. 
\end{lem}
\begin{proof}
We write coordinates in an $(n+1)$ simplex of vertices $(x_0,\dots,x_n)$ as $\sum^*_i \lambda_ix_i$, where $\lambda_i\in [0,1]$ and $\sum\lambda_i=1$. Note that when two of the $x_i$ are equal this is meaningful and lies in a smaller simplex. For $t \in [0,1]$, we define $\gamma_t:\sum^*_i \lambda_ix_i\mapsto \sum^*_i(1-t)\lambda_ix_i + \sum^*_it\lambda_if(x_i)$. We need to check this is meaningful. First, if $(x_i)$ forms a simplex, then $d(x_i,f(x_j))\le d(x_i,x_j)+k$ and $d(f(x_i),f(x_j))\le d(x_i,x_j)+2k$, hence the $x_i$ and $f(x_j)$ together form a simplex in $R_{c+2k}$ (possibly there are some equalities $x_i=f(x_j)$). We also need to check that the definition of $\gamma_t$ matches between different simplices, and this is the case.

To check continuity, we first observe that by definition of the inductive limit topology, it is enough to check continuity on finite subcomplexes (i.e.\ on $F\times [0,1]$ where $F$ ranges over finite subcomplexes of $R_c(X)$). Since such a finite subcomplex maps into a finite subcomplex, this is equivalent to checking the continuity when the simplices are endowed with some standard metric, for instance the Euclidean metric with edges of length 1.

The formula being Lipschitz in all variables $(\lambda_i)$ and $t$, the map $(t,x)\mapsto \gamma_t(x)$ is continuous and defines a homotopy between $\bar{f}$ and $\bar{i}$ (note that the $x_i$ lying on a discrete set, there is no continuity issue at this level).
\end{proof}

For two path-connected metric spaces $Z_1,Z_2$, a continuous map $f: Z_1 \rightarrow Z_2$ is said to be $\pi_1$-injective (respectively $\pi_1$-surjective) if the induced map $f_*$ at the level of fundamental groups is injective (respectively surjective).

Now we fix a base-point on $X$ and we need the following condition on $X$: \begin{center} (*) \textit{there exists $c_0$ such that for all $c\ge c_0$, the inclusion map $R_{c_0}(X)\to R_c(X)$ is $\pi_1$-surjective.} \end{center} This is ensured by some coarse geodesic assumptions such as: for every $c\ge c_0$, for any two points at distance $\le c+1$ there is a third point at distance $\le c$ from both. This in particular holds with $c_0=1$ in a combinatorial connected graph.

We fix another metric space $Y$ that also satisfies (*) (with the same $c_0$). We also fix inverse quasi-isometries $f,g$ between $X$ and $Y$. We can assume (increasing $c_0$ if necessary) that $f:X\to Y$ and $g:Y\to X$ are $(C,Cc_0)$-LS-Lipschitz and $(C',c_0)$-LS-Lipschitz for some $C,C'\ge 1$. We assume all these maps are basepoint-preserving and that $g\circ f$ has distance $\le c_0/2$ to the identity of $X$. (We use nothing on $f\circ g$, hence we just suppose that $X$ is a QI-retract of $Y$.)

For a metric space $Z$, we denote by $\Psi(Z)$ is the set of $r$ such that $R_r(Z)\to R_{2r}(Z)$ is not $\pi_1$-injective. Note that when $Z$ is a combinatorial connected graph, we have $\Psi(Z) \sim \Phi(Z)$, so that the following lemma implies Theorem \ref{thm-qi-inv}.

\begin{lem} \label{prop-qi-inv}
There exists $r_0>0$ and $s\ge 1$ such that if $r\ge r_0$ and $R_{r}(X)\to R_{2r}(X)$ is not $\pi_1$-injective, then $R_{Cs^{-1}r}(Y)\to R_{2Cr}(Y)$ is not $\pi_1$-injective.

In particular, up to finite Hausdorff multiplicative distance, we have the inclusion of $\Psi(X)$ into $\Psi(Y)$. 
\end{lem}

\begin{proof}
Fix a positive number $q$ satisfying $q\le \min((C'C)^{-1},1/2)$.
Fix $r>q^{-1}c_0$ such that $R_r(X)\to R_{2r}(X)$ is not $\pi_1$-injective. 

Assume by contradiction that $R_{Cqr}(Y)\to R_{2Cr}(Y)$ is $\pi_1$-injective. Then we have the following diagram of continuous maps, where the horizontal arrows are standard inclusions and the top square commutes (for each vertical arrow, a simple inequality shows that it is well-defined).

$$\begin{CD}
R_{qr}(X) @> >> R_{2r}(X)\\
@VV\bar{f}V @VV\bar{f}V\\
R_{Cqr}(Y) @> >> R_{2Cr}(Y)\\
@VV\bar{g}V @.\\
R_{CC'qr}(X) @. @.
\end{CD}$$


The upper square induces a commutative square of maps between the fundamental groups. Pick a non-trivial element $\gamma \in \pi_1(R_{r}(X))$ having a trivial image in $\pi_1(R_{2r}(X))$. The assumption (*) implies that $\gamma$ can be chosen to lie in $\pi_1(R_{qr}(X))$. 
 
 Composing (right and then down), we see that $\gamma$ has a trivial image in $\pi_1(R_{Cqr}(Y))$. The injectivity of $R_{Cqr}(Y)\to R_{2Cr}(Y)$ then implies that $\bar{f}(\gamma)\in\pi_1(R_{Cqr}(Y))$ is trivial. Again composing by $g$, it follows that $\bar{g}\circ \bar{f}$ maps $\gamma$ to a trivial element in $\pi_1(R_{C'Cqr}(X))$. On the other hand, by Lemma \ref{distid}, $\bar{g}\circ \bar{f}$, as a map from $R_{qr}(X)$ to $R_{\max(C'Cqr,qr+c_0)}(X)$, is homotopic to the identity. Hence the identity maps $\gamma$ to a trivial element in $\pi_1(R_{\max(C'Cqr,qr+c_0)}(X))$. The upper bound on $q$ implies that $\max(C'Cqr,qr+c_0)\le r$. Hence we obtain that the image of $\gamma$ in $\pi_1(R_{r}(X))$ is trivial, which is a contradiction.
\end{proof}


\section{Relation range of a group} \label{sec-rel-range-group}

\subsection{Definition}

Let $G$ be a group, and $S$ a generating subset of $G$, so that we have a short exact sequence  \[ 1 \longrightarrow N \longrightarrow F_S \longrightarrow G \longrightarrow 1, \] where $F_S$ is the free group on $S$. For every $n \geq 0$, we let $N_n$ be the normal subgroup of $F_S$ generated as such by elements of $N$ or word length at most $n$. By definition $(N_n)$ is an increasing sequence of normal subgroups of $F_S$ ascending to the subgroup $N$. 

\begin{definition}
We denote by $\mathcal{R}_S(G)$ the set of integers $n \geq 0$ such that $N_n$ properly contains $N_{n-1}$. We call $\mathcal{R}_S(G)$ the \textbf{relation range} of the group $G$ with respect to the generating subset $S$.
\end{definition}



\begin{definition} \label{def-dens-rel}
A group $G$ is said to be:
\begin{enumerate}
	\item \textbf{boundedly presented} over $S$ if $\mathcal{R}_S(G)$ is finite;
	\item \textbf{densely related} over $S$ if $\mathcal{R}_S(G) \sim \N$. Otherwise $G$ is said to be \textbf{lacunary presented} over $S$.
\end{enumerate}
\end{definition}

Recall that the Cayley graph $\mathrm{Cay}(G,S)$ of $G$ with respect to $S$ is the graph with $G$ as set of vertices, and $(x,y)$ is an edge if there is $s \in S^{\pm 1}$ such that $y=xs$. Since loops in $\mathrm{Cay}(G,S)$ correspond to relations in $G$, we have the following.

\begin{lem}
$\mathcal{R}_S(G) = \Phi(\mathrm{Cay}(G,S))$.
\end{lem}

Theorem \ref{thm-qi-inv} has the following consequence.

\begin{cor} \label{cor-qi-inv}
Let $G$ be a compactly generated locally compact group, and $S$ a compact generating subset. Then the $\sim$-class of $\mathcal{R}_S(G)$ does not depend on $S$, and is actually a quasi-isometry invariant of the group $G$.

More generally $[\mathcal{R}(H)]\subset [\mathcal{R}(G)]$ whenever $H$ is a large-scale Lipschitz retract of $G$ (e.g.\ a group retract).
\end{cor}


Recall that a \textbf{group retract} of $G$ is a subgroup $H \leq G$ such that there exists a homomorphism $G \to H$ whose restriction to $H$ is the identity.
\begin{definition}
If $G$ is a compactly generated locally compact group, we call $\mathcal{R}(G) = \left[ \mathcal{R}_S(G) \right]$ the \textbf{relation range} of $G$, which does not depend on the choice of compact generating subset $S$. We say that $G$ is \textbf{densely related} if $\mathcal{R}(G) \sim \N$, and \textbf{lacunary presented} otherwise.
\end{definition}

\begin{rem} \label{rem-diff-CT}
Following \cite{CT-Abels}, if $G$ is compactly generated locally compact group, let $\nu(G)$ be the set of non-principal ultrafilters on $\N$ such that the asymptotic cone $\mathrm{Cone}^{\omega}(G,(1/n))$ is simply connected. This is a quasi-isometry invariant, and its complement $\nu(G)^c$ is related in spirit to the relation range. However, it behaves differently: there exist finitely presented groups for which $\nu(G)^c$ can be non-empty, or even be the set of all non-principal ultrafilters. For instance, any lattice $\Gamma$ in the 3-dimensional Lie group $\mathrm{SOL}$ has $\nu(\Gamma)$ empty (that is, $\nu(\Gamma)^c$ is the set of all ultrafilters), in spite of being finitely presented. Hence if we consider a family of groups $(\Lambda_i)$ achieving all possible relation ranges (up to equivalence), then so does the family $(\Gamma\times\Lambda_i)$, but all these groups have $\nu$ empty and thus $\nu$ does not distinguish these groups.
\end{rem}

\begin{rem}[On minimal presentations and relation range] \label{rem-min-RR}
Let $G$ be a finitely generated group with a minimal presentation $G =\langle S\mid R\rangle$, meaning that no relator belongs to the normal subgroup generated by other relators. Then in general there is no relation between the relation range $\mathcal{R}_S(G)$ and the set $\mathcal{L}_R = \{|r|_S: r \in R \}$, in the sense that both inclusions $\left[ \mathcal{R}_S(G)\right] \subset \left[ \mathcal{L}_R\right]$ and $ \left[ \mathcal{L}_R\right] \subset \left[ \mathcal{R}_S(G)\right]$ may fail:
\begin{enumerate}
	\item Let $H$ be a finitely generated densely related group having a minimal presentation $H=\langle S\mid (r_n) \rangle$, and form the group $G = H \ast \Z$. Since $H$ is a retract of $G$, the group $G$ remains densely related by Corollary \ref{cor-qi-inv}. On the other hand $G$ admits the presentation $G = \langle S,a \mid (a^{u_n} r_n a^{-u_n}) \rangle$ for arbitrary $(u_n)$, and this presentation remains minimal. In particular if $(u_n)$ grows very fast, we have a densely related group with a minimal presentation such that $\mathcal{L}_R$ is not $\sim$-equivalent to $\N$. 
	\item Consider the partial presentation of the lamplighter group \[H=\left\langle t,x \, \mid \, [t^{(2n)!}xt^{-(2n)!},x]=1 \, \, \text{for all} \, \, n \geq 1\right\rangle.\] Again form the group $G = H \ast \Z$, whose relation range is equivalent to $\left\{4 (2n)! +4 \right\} \sim \left\{(2n)!\right\}$ (Lemma \ref{lem-graph-prod}). Denoting $r_n = [t^{(2n)!}xt^{-(2n)!},x]$, consider the presentation $G = \langle t,x,a \mid a^{u_n} r_n a^{-u_n} = 1\rangle$, where $u_n$ is chosen so that the length of $a^{u_n} r_n a^{-u_n}$ is equal to $(2n+3)!$ for all $n \geq 1$. This presentation remains minimal, but the $\sim$-class of $\mathcal{L}_R$ is not included inside the relation range of $G$.
\end{enumerate}
\end{rem}

\subsection{Relation range of quotients} \label{subsec-ext}

If $N$ is a normal subgroup of a group $G$ and $S$ generates $G$, we denote by $\mathcal{X}_S(N,G)$ the set of integers $n$ such that there is an element of $N$ of length $n$ outside the normal subgroup of $G$ normally generated by elements of $N$ of length at most $n-1$. 

\begin{prop} \label{prop-bounded-pres-cover}
Let $1 \rightarrow N \rightarrow G \rightarrow Q \rightarrow 1$ be a short exact sequence of groups. Assume that $S$ is a generating subset of $G$ such that $G$ is boundedly presented over $S$. Then $\mathcal{R}_S(Q) = \mathcal{X}_S(N,G)$.
\end{prop}

Before giving the proof, let us derive the following consequence: 

\begin{cor} \label{cor-central-ext-fp}
Let $1 \rightarrow N \rightarrow G \rightarrow Q \rightarrow 1$ be a short exact sequence of groups. If $G$ is finitely presented, then $\mathcal{R}(Q) \sim \mathcal{X}_S(N,G)$.
\end{cor}

Proposition \ref{prop-bounded-pres-cover} will follow from Lemmas \ref{lem-gen-ext} and \ref{lem-gen-ext-boundedly}.


\begin{lem} \label{lem-gen-ext}
Let $1 \rightarrow N \rightarrow G \rightarrow Q \rightarrow 1$ be a short exact sequence of groups, and let $S$ be a generating subset of $G$. Then $\mathcal{X}_S(N,G) \subset \mathcal{R}_S(Q)$.
\end{lem}

\begin{proof}
We denote by $N_Q$ the kernel of $F_S \rightarrow Q$. Let $x \in N$ be an element of length $n \geq 1$ that is not in the normal closure in $G$ of the set of elements of $N$ of length at most $n-1$. Let $w \in F_S$ of length $n$ that represents $x$ in $G$. Then $w$ lies in $N_Q$, and we claim $w$ does not belong to the normal closure in $F_S$ of the set of elements of $N_Q$ of length at most $n-1$.

Argue by contradiction and assume that there is a decomposition $w = \prod \alpha_i r_i \alpha_i^{-1}$, where $r_i \in N_Q$ has length at most $n-1$. If $x_i$ (resp.\ $a_i$) is the element of $G$ represented by $r_i$ (resp.\ $\alpha_i$), then by pushing the above decomposition in $G$ we obtain $x = \prod a_i x_i a_i^{-1}$. Contradiction.
\end{proof}

The following shows that the converse inclusion in the previous lemma also holds when the extension is boundedly presented. 

\begin{lem} \label{lem-gen-ext-boundedly}
Let $1 \rightarrow N \rightarrow G \rightarrow Q \rightarrow 1$ be a short exact sequence of groups. Assume that $S$ is a generating subset of $G$ such that $G$ is boundedly presented over $S$. Then $\mathcal{R}_S(Q) \subset \mathcal{X}_S(N,G)$.
\end{lem}

\begin{proof}
Again we let $N_Q$ (resp.\ $N_G$) be the kernel of $F_S \rightarrow Q$ (resp.\ $F_S \rightarrow G$). Let $r \geq 1$ such that $N_G$ is normally generated by its elements of length at most $r$. Let $n \geq r +1$ such that there is a word $w \in N_Q \subset F_S$ of length $n$, and such that $w$ does not belong to the normal closure of elements of $N_Q$ of length at most $n-1$. We shall prove that $n$ belongs to $\mathcal{X}_S(N,G)$.

We let $x$ be the element represented by the word $w$ in $G$. Then $x \in N$, and we claim that $x$ does not belong to the subgroup normally generated by elements of $N$ of length at most $n-1$. Indeed, assume that in $G$ there is a decomposition $x = \prod a_i x_i a_i^{-1}$, where $a_i \in G$ and $x_i \in N$ has length at most $n-1$. For every $i$, we let $w_i \in F_S$ representing $x_i$ in $G$ and such that $|w_i|_S \leq n-1$, and we choose $\alpha_i \in F_S$ such that $\alpha_i$ represents $a_i$ in $G$. Then $w (\prod \alpha_i w_i \alpha_i^{-1})^{-1} \in N_G$. Since $N_G$ is normally generated by its elements of length at most $r$ and $n-1 \geq r$ by assumption, we deduce that $w$ belongs to the normal closure of the set of elements $t \in N_Q$ such that $|t|_S \leq n-1$. Contradiction. 
\end{proof}


\subsection{Lacunary approximations}

Given a locally compact group $G$, we call approximation of $G$ a sequence of continuous surjective homomorphisms with discrete kernel $G_0 \to G_1 \to \dots$ and $G_i \to G$ such that the obvious diagrams commute, and such that the kernel $N$ of $G_0\to G$ is equal to the union of kernels $N_i$ of $G_0\to G_i$, and such that all $G_i$ are compactly presented.

Given such an approximation, fix a compact generating subset of $G_0$. For $i\ge 1$, let $\rho_i\in [1,+\infty]$ be the smallest length of an element in $N \smallsetminus N_{i}$. Let $s_i$ be the smallest number $s$ such that $N_i$ is normally generated by the intersection of $N_i$ with the $s$-ball in $G_0$; then $s_i<\infty$ because $G_i$ is compactly presented.

We say that the approximation is {\bf lacunary} if $\rho_{i}/s_i\to +\infty$.

\begin{prop}\label{lpla}
A compactly generated locally compact group is lacunary presented if and only if it admits a lacunary approximation. Moreover, in the discrete case, $G_0$ can chosen to be free.
\end{prop}
\begin{proof}
This is easy. Start from any compactly presented group $G_0$ such that $G$ can be written as a quotient of $G_0$ with discrete kernel $N$ (we can choose $G_0$ free in the discrete case); fix a compact generating subset of $G_0$. First define $H_n$ as the quotient of $G_0$ with the normal closure $N_n$ of $N\cap B(n)$, where $B(n)$ is the $n$-ball in $G_0$. Then $N=\bigcup N_n$ (ascending union).

Assume now that that $G$ is lacunary presented. Then for every integers $c,n_0\ge 1$, we can find $n\ge n_0$ such that $[n,cn]$ contains no element of the relation range. Therefore, we can extract from $(H_n)$ a subsequence satisfying the required conditions.

Conversely, the existence of a lacunary approximation implies that there is no new relation of size between $s_i$ and $\rho_{i}$ and since the ratio $\rho_{i}/s_i$ tends to infinity, this implies that the relation range is not at finite multiplicative distance to $\N$.
\end{proof}

Let us provide an application of lacunary approximations to the Hopfian property, extending the recent result by Coulon and Guirardel \cite{CG} that lacunary hyperbolic groups are Hopfian. Recall that a group is Hopfian if all its surjective endomorphisms are injective.

\begin{prop} \label{prop-lac-approx-hopf}
Let $G$ be a finitely generated group with a lacunary approximation by Hopfian groups. Then $G$ is Hopfian.
\end{prop}

\begin{proof}
First, we can find an epimorphism from a free group $G_{-1}$ to $G_0$ and then shift indices to assume that $G_{0}$ is free. Next, let $f$ be a surjective endomorphism of $G$. Since $G_0$ is free, we can lift it to an endomorphism $\hat{f}$ of $G_0$; then $\hat{f}$ maps the 1-ball into the $k$-ball for some $k$. We can describe $G_i$ as the quotient of $G_0$ by the normal closure of a certain subset $R_i$ of the $s_i$-ball. Then $\hat{f}(R_i)$ belongs to the $ks_i$-ball and also belongs to the kernel of $G_0\to G$. If $i$ is large enough, say $i\ge i_0$, so that $\rho_{i}>ks_i$, it follows that $\hat{f}(R_i)$ belongs to the kernel of $G_0 \to G_i$. Thus for $i\ge i_0$, $\hat{f}$ factors to an endomorphism $f_i$ of $G_i$, still lifting $f$. Then since $f$ is surjective, we can write generators as elements of the image in $G$, and lift this to $G_i$ for, say $i\ge i_1\ge i_0$. Thus $f_i$ is a surjective endomorphism for all $i\ge i_1$. Since $G_i$ is Hopfian, it follows that $f_i$ is an automorphism for all $i\ge i_1$, and therefore $f$ is an automorphism as well.
\end{proof}

\begin{cor}[Coulon-Guirardel \cite{CG}]
Finitely generated lacunary hyperbolic groups are Hopfian.
\end{cor}

\begin{proof}
Sela proved that torsion-free hyperbolic groups are Hopfian  \cite{Sela-hopf}, and this was more recently extended to arbitrary hyperbolic groups by Reinfeldt and Weidmann in the preprint \cite{Rein-Weid}. Given that finitely generated lacunary hyperbolic groups admit lacunary approximations by hyperbolic groups \cite[Th.\ 3.3]{OOS}, we conclude by Proposition \ref{prop-lac-approx-hopf}.
\end{proof}

Proposition \ref{prop-lac-approx-hopf} can also be applied beyond the lacunary hyperbolic case. For instance, those partial finite presentations of $\Z\wr\Z$ are easily seen to be residually finite (and hence Hopfian), so the lacunary presented groups obtained in this way are Hopfian groups.

\subsection{Asymptotic cones of densely related groups} \label{subsec-cones}

The following terminology and notation are essentially borrowed from \cite{Papasoglu}. We let $\mathbf{I}_2 = [0,1] \times [0,1]$ be the unit Euclidean square of dimension two, and denote by $\partial \mathbf{I}_2$ its boundary. A collection of squares $D_1, \ldots, D_k$ is defined to be a partition of $\mathbf{I}_2$ if $D_i \cap D_j = \partial D_i \cap \partial D_j$ whenever $i \neq j$, and if $\mathbf{I}_2$ is the union of the squares $D_i$.

If $X$ is a geodesic metric space, a loop in $X$ is by definition a continuous map $\alpha : \partial \mathbf{I}_2 \rightarrow X$, and we freely identify a loop with its image in $X$. A partition $\pi$ of $\alpha$ is a continuous map extending $\alpha$ to $\partial D_1 \cup \ldots \cup \partial D_k$, where $D_1, \ldots, D_k$ is a partition of $\mathbf{I}_2$. We define the \textbf{mesh} of $\pi$ as the maximal length of the paths $\pi(\partial D_i)$.

\begin{lem} \label{lem-limit-loops}
Let $X$ be a geodesic metric space, and let $\mathcal{C} = \mathrm{Cone}^{\omega}(X,(x_n),(s_n))$ be an asymptotic cone of $X$. Then
\begin{enumerate}[label=(\alph*)]
\item any loop in $\mathcal{C}$ is the $\omega$-limit of a sequence of loops in $X$;
\item if $\alpha$ is a loop in $\mathcal{C}$ which is the $\omega$-limit of a sequence of loops $(\alpha_n)$ in $X$, then any partition of $\alpha$ is the $\omega$-limit of a sequence of partitions of $\alpha_n$.
\end{enumerate}
\end{lem}

\begin{proof}
Statement $(a)$ is proved in \cite[Prop.\ 2.2]{Kent} for paths rather than loops, but the proof can be easily adapted to realize any loop in $\mathcal{C}$ as the $\omega$-limit of a sequence of loops in $X$.

Statement $(b)$ is obtained similarly, working in each square of the partition.\qedhere
\end{proof}

Recall that by a theorem of Gromov \cite[\S 5.F]{Gro-asy}, if a finitely generated group $G$ has all its asymptotic cones simply connected, then $G$ is finitely presented (and has a polynomially bounded Dehn function). One cannot hope to obtain the same conclusion if we weaken the hypothesis by requiring that \textit{one} asymptotic cone of $G$ is simply connected, as for example any non-hyperbolic lacunary hyperbolic group is infinitely presented \cite[Appendix]{OOS}. Recall that a group is \textbf{lacunary hyperbolic} if it admits (at least) one asymptotic cone that is a real tree \cite{OOS}.

We next show that if $G$ has one asymptotic cone that is simply connected, then the group $G$ is lacunary presented. The proof follows the same strategy as the proof of the direct implication of Theorem 4.4 in \cite{Drutu-cone}, which says that simple connectedness of all asymptotic cones implies a certain division property for loops (see \cite{Drutu-cone} for the relevant definition).

\begin{prop} \label{prop-1-cone-sc}
Let $G$ be a compactly generated locally compact group, and $(s_n)$ a sequence of positive numbers. Assume that the asymptotic cone $\mathrm{Cone}^{\omega}(G,(s_n))$ is simply connected for some ultrafilter $\omega$ such that $\lim_\omega s_n=+\infty$. Then $\{s_n:n\in\N\} \not\subset \RR$. 
\end{prop}

\begin{proof}
Let $S$ be a compact generating subset of $G$. We argue by contradiction, and assume that $\{s_n:n\in\N\} \subset \RR$. This implies that there exist a constant $c > 0$ and a sequence of relations $r_n \in F_S$ so that \[ c^{-1} s_n \leq |r_n|_S \leq c s_n,\] and $r_n$ is not generated by relations of smaller length. By construction, the sequence of loops $\alpha_n: \partial \mathbf{I}_2 \rightarrow \mathrm{Cay}(G,S)$ parametrized proportionally to the length associated to $r_n$ yields a loop $\alpha: \partial \mathbf{I}_2 \rightarrow \mathrm{Cone}^{\omega}(G,(s_n))$. Since $\mathrm{Cone}^{\omega}(G,(s_n))$ is supposed to be simply connected, the map $\alpha: \partial \mathbf{I}_2 \rightarrow \mathrm{Cone}^{\omega}(G,(s_n))$ can be extended to a continuous function $\sigma : \mathbf{I}_2 \rightarrow \mathrm{Cone}^{\omega}(G,(s_n))$. Now since $\mathbf{I}_2$ is compact, the map $\sigma$ is uniformly continuous, and there exists $\eta > 0$ so that $d_{\omega}(\sigma(t),\sigma(u))$ is at most $c^{-1}/5$ as soon as the distance between $t$ and $u$ is at most $\eta$. Let us consider the partition of $\mathbf{I}_2$ given by the net \[ \left\{(a \eta, b \eta): a,b \in \Z, \, 0 \leq a,b \leq 1/ \eta \right\}.\] Since the mesh of this partition is equal to $4 \eta$, the restriction of $\sigma$ to this partition yields a partition of the loop $\alpha$ in $\mathrm{Cone}^{\omega}(G,(s_n))$ of mesh at most $4c^{-1}/5$.

From a geometric point of view, the fact that $r_n$ does not belong to the normal subgroup of $F_S$ generated by relations of length at most $|r_n|_S -1$ implies that the mesh of any partition of the loop $\alpha_n$ is at least $|r_n|_S$, so in particular at least $c^{-1} s_n$. Now we claim that this implies that the mesh of any partition of the loop $\alpha$ is at least $c^{-1}$. Indeed according to Lemma \ref{lem-limit-loops}, any partition $\pi$ of the loop $\alpha$ is the $\omega$-limit of a sequence of partitions $\pi_n$ of the loop $\alpha_n$ in $\mathrm{Cay}(G,S)$. Being the limit over the ultrafilter $\omega$ of the mesh of $\pi_n$ rescaled by $s_n$, the mesh of $\pi$ is at least $c^{-1}$ by the previous observation. This readily gives a contradiction with the previous paragraph.
\end{proof}

\begin{cor} \label{cor-full-pres-cones}
Let $G$ be a densely related compactly generated group. Then none of the asymptotic cones of $G$ are simply connected.
\end{cor}

Since a real tree is a simply connected metric space, Corollary \ref{cor-full-pres-cones} also admits the following consequence, which can also be derived from Proposition \ref{lpla} using the combinatorial characterization of lacunary hyperbolic groups from \cite{OOS} in the discrete case, and from \cite{LB-lac-hyp} in the locally compact case.

\begin{cor} \label{cor-lac-hyp-lac-rel}
Any lacunary hyperbolic locally compact group is lacunary presented.
\end{cor}

\begin{rem}\label{comparisonLHLP}
The class of lacunary presented groups is much larger than the class of lacunary hyperbolic groups. For instance, in the discrete case, it contains all finitely presented groups and is stable under direct powers and direct products with finitely presented groups (Corollary \ref{stability}), which are essentially never lacunary hyperbolic (recall that being direct limits of hyperbolic groups, lacunary hyperbolic groups contain no copy of $\Z^2$).
\end{rem}

\begin{rem}\label{rem-linear}
(See also Question \ref{linear-question}) We do not know any example of a finitely generated linear group that is neither finitely presented nor densely related. If it is true that every finitely generated infinitely presented linear group is densely related, then this implies according to Corollary \ref{cor-lac-hyp-lac-rel} that every finitely generated lacunary hyperbolic group that is linear is actually a hyperbolic group (because a finitely presented lacunary hyperbolic group is hyperbolic \cite[Appendix]{OOS}). Whether this last assertion is true was asked in \cite{OOS}.
\end{rem}

\section{Specific classes of groups and examples} \label{sec-onto-Z}

\subsection{Classes of groups with various relation ranges} \label{subsec-various-RR}

\subsubsection{Lacunary hyperbolic groups}

Recall from Corollary \ref{cor-lac-hyp-lac-rel} that every finitely generated lacunary hyperbolic group is lacunary presented. The following proposition, the proof of which is a combination of works of Bowditch \cite{Bowditch-RR} and Olshanskii-Osin-Sapir \cite{OOS}, says that this is actually the only restriction on the relation range of a lacunary hyperbolic group. 

\begin{prop}
For every subset $I$ of $\N$ with $I\nsim\N$, there exists a lacunary hyperbolic group $\Gamma$ with $\mathcal{R}(\Gamma)\sim I$.
\end{prop}

\begin{proof}
Upon changing $I$ into some $I' \sim I$, we may easily find a presentation $\Gamma = \left\langle a,b \, | \, (r_n)_n\right\rangle$ so that $\{|r_n|: n \in \N\} = I'$ and satisfying the $C'(1/7)$ small cancelation condition. This last property means that $\{r_n: n \in \N\}$ is a set of reduced words stable under taking cyclic conjugates, and such that for every $n \in \N$ and every $m \neq n$, the largest common prefix of $r_n$ and $r_m$ has length at most $(1/7)|r_n|$. The $C'(1/7)$ condition implies on the one hand that the relation range of the group $\Gamma$ is given by the set of lengths of relators \cite[Lem.\ 5]{Bowditch-RR}, i.e.\ $\mathcal{R}(G) \sim I' \sim I$, and on the other hand that $\Gamma$ is lacunary hyperbolic since $I' \nsim \N$ \cite[Prop.\ 3.12]{OOS}.
\end{proof}

\subsubsection{Solvable groups} \label{subsec-abels}

If $G$ is a group and $\ell: G \rightarrow \R_+$ is a length function on $G$, we denote by $\mathcal{X}_{\ell}(G)$ the set of integers $n$ such that there is $g \in G$ such that $\ell(g) = n$ and $g$ does not belong to the subgroup of $G$ generated by elements of of length at most $n-1$. Note that if $Z$ is a central subgroup of a group $G$ and $S$ is a generating subset of $G$, then we have $\mathcal{X}_{\ell_S}(Z) = \mathcal{X}_S(Z,G)$ (see \S\ref{subsec-ext} for the definition of $\mathcal{X}_S(Z,G)$). 

The proof of the following lemma is routine, and we omit it.

\begin{lem} \label{lem-bilip-lengths}
If $\ell_1, \ell_2$ are bi-Lipschitz equivalent length functions on $G$, then $\mathcal{X}_{\ell_1}(G) \sim \mathcal{X}_{\ell_1}(G)$. 
\end{lem}

In the sequel we denote by $R$ the ring $\mathbf{F}_p[t,t^{-1},(t+1)^{-1}]$, and we consider Abels' group $A_4(R)$ consisting of upper triangular matrices of $\mathrm{GL}_4(R)$ whose first and last diagonal entries are equal to $1$ (see \cite{CT-Abels}).


For $i \geq 0$, we denote by $z_i \in A_4(R)$ the element whose diagonal entries are equal to $1$, and whose only non-zero off-diagonal entry has coordinate $(1,4)$ and is equal to $t^i$. We note that every $z_i$ lies in the center of $A_4(R)$. If $I \subset \N$, we let $Z_I = \left\langle z_i \right\rangle_{i \in I}$. The group $Z_I$ is abelian and isomorphic to the group $\mathbf{F}_p^{(I)}$. Every $z \in Z_I$ can be uniquely represented as $\sum_{k \geq 0} a_k t^k$, where all but finitely many $a_k \in \mathbf{F}_p$ are equal to zero. We consider the length function on $Z_I$ defined by $\ell(z) = \sup\left\{k : a_k \neq 0 \right\}$. 

\begin{lem} \label{lem-X-Z_I}
We have $\mathcal{X}_{\ell}(Z_I) = I$.
\end{lem}

\begin{proof}
We have $\ell(z_i) = i$ for every $i \in I$, and clearly $z_i$ does not belong to the subgroup generated by the $z_k$ for $k < i$. This shows $I \subset \mathcal{X}_{\ell}(Z_I)$, and the converse is clear since $\ell: Z_I \rightarrow \R_+$ take values in $I$.
\end{proof}

Recall that Grigorchuk gave, using growth exponents, a continuum of pairwise non quasi-isometric finitely generated groups of intermediate growth \cite{Grig-continuum}. In view of Corollary \ref{cor-qi-inv}, the following proposition provides a continuum of quasi-isometry classes of finitely generated solvable groups, a fact that was established in \cite{CT-Abels} using asymptotic cones. This was also obtained in \cite{Bri-Zhe} using compression of embeddings into $\mathbf{L}^p$-spaces. 

\begin{prop} \label{prop-RR-abels}
Let $I \subset \N$, and let $Q_I$ be the quotient of the group $A_4(R)$ by its central subgroup $Z_I$. Then $\mathcal{R}(Q_I) \sim I$.
\end{prop}

\begin{proof}
Let $S$ be a finite generating subset of the group $A_4(R)$. Since the group $A_4(R)$ is finitely presented \cite[Th.\ 5.1]{CT-Abels}, we have $\mathcal{R}_S(Q_I) = \mathcal{X}_S(Z_I,G)$ according to Corollary \ref{cor-central-ext-fp}. Now since $Z_I$ is central in $A_4(R)$, if follows that $\mathcal{X}_S(Z_I,G) = \mathcal{X}_{\ell_S}(Z_I) = \mathcal{X}_{\ell}(Z_I)$, where the last equality follows from Lemma \ref{lem-bilip-lengths} (since the metrics $\ell$ and $\ell_S$ are bi-Lipschitz equivalent on $Z_I$, see for instance \cite[$\S$4.3]{CT-Abels}). Therefore $\mathcal{R}_S(Q_I) = \mathcal{X}_{\ell}(Z_I)$, and the conclusion then follows from Lemma \ref{lem-X-Z_I}.
\end{proof}

\begin{rem}\label{abelscon}
This completes a remark in the introduction. In \cite{BCGS}, the class of extrinsic condensation group is studied: these are finitely generated groups $G$ such that for every finitely presented group $H$ and surjective homomorphism $p:H\to G$, the kernel of $p$ contains uncountably many normal subgroups of $H$. This property is a strengthening of being infinitely presented but is not a quasi-isometry invariant, and is actually not even closed under taking finite index overgroups.

For instance, if $B$ is Abels' group (denoted $A_4/Z$ in in \cite[Example 5.13]{BCGS}); here it is rather, for some prime $p$, the quotient of $A_4(\Z[1/p])$ by its center $Z\simeq \Z[1/p]$, then $B\times B$ is an extrinsic condensation group but its overgroup of index 2 $B\wr \Z/2\Z$ is not.
\end{rem}

\subsection{Graph products and wreath products}

We will make use of the following easy lemma.

\begin{lem} \label{lem-retract-sys-gen}
Let $G$ be a group, $H$ a retract of $G$ and $\varphi: G \twoheadrightarrow H$ be a homomorphism whose restriction to $H$ is the identity. Assume that $S_G$ is a generating subset of $G$, and write $S_H = \varphi(S_G)$. Then $\mathcal{R}_{S_H}(H) \subset \mathcal{R}_{S_G}(G)$.
\end{lem}

\begin{proof}
Let $w \in F_{S_H}$ be a relation in the group $H$. If $w$ admits a decomposition in $F_{S_G}$ as a product of conjugates of elements of length $\leq n$, then by replacing each letter by its image by $\varphi$ we obtain a decomposition of $w$ in $F_{S_H}$ as a product of conjugates of elements of length $\leq n$. It follows that if $w$ is not generated by relations of smaller length in $F_{S_H}$, then the same holds in $F_{S_G}$, and the result is proved.
\end{proof}

We derive the following consequence about group retracts, which also follows from Theorem \ref{thm-qi-inv}.

\begin{cor} \label{cor-retract}
Let $G=N\rtimes Q$ be a semidirect product decomposition of locally compact groups, with $G,Q$ compactly generated. Then $\mathcal{R}(Q)\subset\mathcal{R}(G)$. In particular, if $G$ is lacunary presented then so is $Q$ (or equivalently, if $Q$ is densely related then so is $G$).
\end{cor}

Let $X = (V,E)$ be a graph, and let $(G_v)_{v \in V}$ be a family of groups indexed by the set of vertices of $X$. Recall that the graph product $P$ associated to this data is the quotient of the free product of all $G_v$ by the relations $[G_v,G_w]=1$ whenever $(v,w)$ is an edge of $X$. 

\begin{lem} \label{lem-graph-prod}
Assume that $S_v$ is a generating subset of $G_v$ for every $v \in V$, and let $S$ be the disjoint union of all $S_v$. Then \[ \mathcal{R}_S(P) = \bigcup_{v \in V} \mathcal{R}_{S_v}(G_v). \]
\end{lem}

\begin{proof}
Each $G_v$ is a retract of $P$, so by Lemma \ref{lem-retract-sys-gen} we have $\mathcal{R}_{S_v}(G_v) \subset \mathcal{R}_S(P)$ for every $v \in V$. Now let $w$ be a word in the elements of $S$ of length $\geq 5$, such that $w$ is a relation in $P$ that is not generated by relations of smaller length. Remark that in $P$ every relation that is not generated by relations of length at most $4$ comes from relations in the $G_v$. Since moreover $w$ is not generated by relations of smaller length $w$ must be a relation in one of the $G_v$, which proves the converse inclusion.
\end{proof}

\begin{cor}\label{stability}
Let a compactly generated locally compact group $G$ be a graph product of finitely many groups $G_1, \dots, G_n$. Then $\mathcal{R}(G)\sim\bigcup\mathcal{R}(G_i)$. In particular, if all $G_i$ are lacunary presented then so is $G$. In particular again, the class of lacunary presented groups is closed under taking direct or free products with compactly presented locally compact groups, and under taking direct or free powers ($G\mapsto G^k$ or $G^{\ast k}$).
\end{cor}

\begin{rem}
Similarly to what happen for lacunary hyperbolic groups \cite[Example 3.16]{OOS}, the class of lacunary presented groups is not closed under direct or free products. For instance, consider the subset $A_1=\bigcup_{n\ge 1}[(2n)!,(2n+1)!]$ (where intervals are understood to be within integers), and let $A_2$ be the complement of $A_1$ in $\N$. Then there exist finitely generated groups $\Gamma_1,\Gamma_2$ with $\mathcal{R}(\Gamma_i)\sim A_i$ (see the introduction, or \S\ref{subsec-abels}). Then both $\Gamma_1$ and $\Gamma_2$ are lacunary presented, but using Corollary \ref{stability}, $\Gamma_1\times\Gamma_2$ and $\Gamma_1\ast\Gamma_2$ are densely related. 
\end{rem}

Recall that a standard wreath product $H \wr G$, with $H$ non-trivial and $G$ infinite, is never finitely presented \cite{Bau-61}. The following strengthens this result by showing that such groups are actually densely related.

\begin{prop} \label{prop-wreath-dr}
Let $G,H$ be finitely generated groups such that $H$ is non-trivial and $G$ is infinite. Then the standard wreath product $H \wr G$ is densely related.
\end{prop}

\begin{proof}
Let $S$ be the union of finite generating subsets $S_G$ and $S_H$ of respectively $G$ and $H$. For every $n \geq 0$, let us consider the graph structure $X_n$ on $G$ defined by putting an edge between $g_1$ and $g_2 \neq g_1$ if and only if $d_{S_G}(g_1,g_2) \leq n$. We denote by $P_n$ the graph product associated to $X_n$ for which all groups are equal to $H$. Since the action of $G$ on itself preserves the graph structure $X_n$, we can consider the semidirect product $\Gamma_n = P_n \rtimes G$. By construction we have a surjective homomorphism $\Gamma_{n} \twoheadrightarrow \Gamma_{n+1}$ for every $n \geq 0$. Since $(X_n)$ converges to the complete graph on $G$, the direct limit of the sequence $(\Gamma_n)$ is exactly the wreath product $H \wr G$.

Fix a non-trivial $s \in S_H$. For every $n \geq 1$, we choose some element $g_n \in G$ whose length with respect to $S_G$ is exactly $n$ (such an element always exists because $S_G$ is finite and $G$ is infinite). Choose a word $v_n$ in the elements of $S_G$ of length $n$ representing $g_n$ in the group $G$, and write $w_n = [v_n s v_n^{-1},s]$. By construction $w_n$ is a relation of length $4n+4$ in $H \wr G$. Since there is no edge in $X_{n-1}$ between $1_G$ and $g_n$, the elements $g_n s g_n^{-1}$ and $s$ generate their free product in $P_{n-1}$ \cite[Lem.\ 2.3 (2)]{Cor-wp}. In particular the word $w_n$ is not trivial in $\Gamma_{n-1}$. Now by remarking that every relation in $H \wr G$ of length at most $4n$ is a relation in $\Gamma_{n-1}$, we obtain that $w_n$ cannot be generated by relations of length $\leq 4n$. Therefore $\mathcal{R}_S(H \wr G)$ contains an element between $4n$ and $4n+4$, and the proof is complete.
\end{proof}

\subsection{Iterations of endomorphisms} \label{subsec-non-hopf}

Following \cite{Grig-Mama}, we consider the following situation. Let $G$ be a group with a finite index normal subgroup $H$, and denote by $X = \left\{x_1,\ldots,x_r\right\}$ a system of coset representatives, so that $G = x_1 H \cup \ldots \cup x_r H$. Assume we are given a set $\Phi = \left\{\varphi_i\right\}$ of surjective homomorphisms \[ \varphi_i: H \rightarrow G, \, \text{for} \, \, i=1,\ldots,r,\] with the following properties:

\begin{enumerate}
	\item \label{hyp1} $\bigcap_{i=1}^r \ker(\varphi_i) \neq 1$.
	\item \label{hyp2} for every $i,j \in \left\{1,\ldots,r\right\}$, there exists $k \in \left\{1,\ldots,r\right\}$ such that \[\varphi_i \circ \mu_{x_j} = \varphi_k,\] where $\mu_{x_j}: H \rightarrow H$ is the automorphism of $H$ induced by the conjugation by $x_j$.
	\item \label{hyp3} The image of $H \rightarrow G^r$, $h \mapsto (\varphi_1(h),\ldots,\varphi_r(h))$, contains the diagonal $\Delta(G)$ in $G^r$.
\end{enumerate}


\begin{rem}
It would actually be enough for our purpose that the image of $H \rightarrow G^r$ contains the diagonal $\Delta(H)$ in $H^r \leq G^r$.
\end{rem}

From now on we fix $G$, $H$ and $\Phi$ satisfying (1),(2) and (3). We define inductively a sequence of subgroups by \[ K_0 = 1, \, \text{and} \, \, K_{n+1} = \bigcap_{i=1}^r \varphi_i^{-1}(K_n) \] for all $n \geq 0$.

\begin{lem} \label{lem-seq-Kn}
$K_n \leq H$ is a normal subgroup of $G$, and $K_n$ is properly contained in $K_{n+1}$ for all $n \geq 0$.
\end{lem}

\begin{proof}
That $K_n$ is a subgroup of $H$ and $K_n \leq K_{n+1}$ follow by an easy induction. Let us check that $K_n$ is normal in $G$. We also proceed by induction, the case $n=0$ being trivial. Assume that $K_n$ is normal in $G$, and choose $h_{n+1} \in K_{n+1}$ and $g \in G$. Given $i \in \left\{1,\ldots,r\right\}$, we want to see that $\varphi_i(g h_{n+1} g^{-1}) \in K_n$. Let $j$ and $h \in H$ be such that $g = x_j h$. We have: \[ \varphi_i(g h_{n+1} g^{-1}) = \varphi_i \circ \mu_{x_j} (h h_{n+1} h^{-1}) = \varphi_k(h h_{n+1} h^{-1}),\] where the integer $k$ is provided by (2). Therefore we have \[ \varphi_i(g h_{n+1} g^{-1}) = \varphi_k(h) \varphi_k (h_{n+1}) \varphi_k(h)^{-1},\] and $\varphi_k (h_{n+1}) \in K_n$ since $h_{n+1} \in K_{n+1}$. Since $K_n$ is normal in $G$ by assumption, we conclude that $\varphi_i(g h_{n+1} g^{-1})$ belongs to $K_n$.

To prove that the inclusion $K_n \subset K_{n+1}$ is proper for all $n$, assume by contradiction that there is $n \geq 1$ such that $K_{n} = K_{n+1}$, and take $g \in K_n$. According to (3) we may find $h \in H$ such that $\varphi_i(h) = g$ for every $i$, which implies that $h$ belongs to $K_{n+1}$ since $g \in K_n$. Therefore $h \in K_{n}$ by our assumption, and it follows that the element $g$ belongs to $K_{n-1}$. So $K_{n-1} = K_{n}$. By repeating the argument we obtain that $K_0 = K_1$, which is a contradiction with (1).
\end{proof}

We denote by \[ K_{\Phi} = \bigcup_{n \geq 0} K_n\] the increasing union of the subgroups $K_n$, which is therefore a normal subgroup of $G$, and by $G_\Phi = G / K_\Phi$ be the associated quotient. In other words, we have a sequence of groups and surjective homomorphisms \[ G \stackrel{\pi_0}{\longrightarrow} G/K_1 \stackrel{\pi_1}{\longrightarrow} \cdots \longrightarrow G/K_n \stackrel{\pi_n}{\longrightarrow} \cdots \] whose direct limit is $G_{\Phi}$.


In the sequel we assume moreover that $G$ is a compactly generated locally compact group, that $H$ is closed in $G$ and that all $\varphi_i: H \rightarrow G$ are continuous. Note that the continuity of the $\varphi_i$ implies that the $K_n$ are closed subgroups of $G$, but it may happen that $K_{\phi}$ is not closed (as the example $H = G=\mathbf{S}^1$ and $\varphi(z) = z^2$ shows, for which $K_{\Phi}$ is dense). From now on we will assume that $K_{\phi}$ is also closed in $G$, and hence $G_{\Phi}$ is a locally compact group.

\begin{rem} \label{rem-K-closed}
If there is a neighbourhood of the identity $U$ such that $U \cap \ker (\varphi_i) = 1$ and $\varphi_i(U) \subset U$ for every $i$, then $K_{\Phi}$ is a discrete subgroup of $G$.
\end{rem}

Since $K_{\Phi}$ is the increasing union of its closed subgroups $K_n$, by the Baire category theorem the subgroups $K_n$ must eventually be open in $K_{\Phi}$. In particular for every $k > 0$, the intersection between $K_{\Phi}$ and $S^k$ is compact in $K_{\Phi}$, and therefore must be contained in some subgroup $K_{n_k}$. Since $K_{n_k}$ is a normal subgroup of $G$, this shows that the normal subgroup of $G$ generated as such by elements of $K_{\Phi}$ of length at most $k$ is contained in $K_{n_k}$. Since all the inclusions $K_n \subset K_{n+1}$ are strict (Lemma \ref{lem-seq-Kn}), this shows that $K_{\Phi}$ is not compactly generated as a normal subgroup of $G$, and therefore the group $G_{\Phi}$ is not compactly presented by Lemma \ref{lem-gen-ext}.


\begin{thm} \label{thm-lim-fully}
Let $G$ be a compactly generated locally compact group, and assume that $\Phi = \left\{\varphi_i\right\}$ satisfy (1), (2) and (3). Then the group $G_{\Phi} = G / K_{\Phi}$ is densely related.
\end{thm}

We will use the following lemma, the proof of which is standard.

\begin{lem} \label{lem-hyp3-metric}
Let $S$ be a compact generating subset of $G$. Then there exists $c >0$ such that for every $g \in G$, there exists $h \in H$ such that $\varphi_i(h) = g$ for every $i$ and $|h|_S \leq c |g|_S$.
\end{lem}

\begin{proof}
The image of $\varphi: H \rightarrow G^r$, $h \mapsto (\varphi_1(h),\ldots,\varphi_r(h))$, contains the diagonal in $G^r$ thanks to (3). Let $\Sigma \subset H$ be a compact subset of $H$ such that $\varphi(\Sigma) = S$ (where we identify $S$ and its diagonal embedding), and let $c = \sup |\sigma|_S$, for $\sigma \in \Sigma$. Since $\Sigma$ is compact, $c$ is finite. Now let $g \in G$, and let $\ell = |g|_S$, so that there exist $s_1,\ldots,s_\ell \in S$ such that $g = s_1 \cdots s_\ell$. For all $k$, let $\sigma_k \in \Sigma$ such that $\varphi(\sigma_k) = s_k$, and write $h = \sigma_1 \cdots \sigma_\ell$. Then $h$ satisfies the conclusion, because each $\varphi_i$ maps $h$ to $g$, and $|h|_S \leq \sum |\sigma_i|_S \leq c \ell = c |g|_S$.
\end{proof}

\begin{proof}[Proof of Theorem \ref{thm-lim-fully}]
Let $S$ be a compact generating subset of $G$, and $\pi: F_S \twoheadrightarrow G$ the canonical projection. For $k \geq 0$, let $N^{(k)} = \pi^{-1}(K_k)$. Note that $N^{(k)}$ is a normal subgroup of $F_S$ in view of Lemma \ref{lem-seq-Kn}. The increasing union $N^{(\infty)}$ of the subgroups $N^{(k)}$ is nothing but the kernel of the natural map from $F_S$ to the group $G_{\Phi}$: \[ 1 \longrightarrow N^{(\infty)} \longrightarrow F_S \longrightarrow G_{\Phi} \longrightarrow 1. \] For every $k \geq 1$ we let $N^{(\infty)}_k$ be the normal subgroup of $F_S$ generated as such by elements of $N^{(\infty)}$ of length at most $k$. 

Fix an integer $k \geq 1$. Denote by $n_k$ the smallest integer so that the set of elements of $N^{(\infty)}$ of word length at most $k$ lies in $N^{(n_k)}$. The discussion before the proposition shows that such an integer $n_k$ always exists. Note that since $N^{(n_k)}$ is a normal subgroup, we have $N^{(\infty)}_{k} \subset N^{(n_k)}$.

By definition of $n_k$, there must exist a word $w$ in $N^{(\infty)} \smallsetminus N^{(n_k-1)}$ of length at most $k$. Set $g = \pi(w)$. According to (3), there exists $h \in H$ such that $\varphi_i(h) = g$ for every $i$. Let $w' \in F_S$ such that $h = \pi(w')$. Then we claim that $w' \in N^{(\infty)}$. Indeed, for every $i$ we have $\varphi_i(\pi(w')) = \varphi_i(h) = g = \pi(w)$, so if $n$ is such that $w \in N^{(n)}$, then $w' \in N^{(n+1)}$. On the other hand the word $w'$ does not belong to $N^{(n_k)}$, because otherwise we would have $w \in N^{(n_k-1)}$. A fortiori $w'$ does not belong to $N^{(\infty)}_k$ since $N^{(\infty)}_{k} \subset N^{(n_k)}$. This implies that the relation range $\mathcal{R}_S(G_{\Phi})$ contains an element between $k$ and $\left| w' \right|_S$. 

Now we may plainly choose $w' \in F_S$ such that $|w'|_S = |h|_S$, and according to Lemma \ref{lem-hyp3-metric} we may also choose $h \in H$ such that $|h|_S \leq c |g|_S$ for some constant $c$ depending only on $S$. But $g= \pi(w)$, so the word length of $g$ with respect to $S$ is at most $|w|_S$. Since $|w|_S \leq k$ by definition of $w$, it follows that $\left| w' \right|_S \leq ck$. Therefore $\mathcal{R}_S(G_{\Phi}) \cap [c,ck] \neq \emptyset$, and the statement is proved.
\end{proof}


\subsubsection{The Grigorchuk group}

We consider the Grigorchuk group $\mathfrak{G}$ introduced in \cite{grig80}. 

\begin{thm} \label{thm-grigorchuk-dens-rel}
The Grigorchuk group $\mathfrak{G}$ is densely related.
\end{thm}

\begin{proof}
Let $G = \left\langle a,b,c,d \, | \, a^2=b^2=c^2=d^2=bcd=1\right\rangle$ be the free product $C_2 \ast (C_2 \times C_2)$, and $H$ be its subgroup of index two which projects trivially on the first $C_2$. $H$ is generated by $b,c,d$ and their conjugates by $a$. The assignments \[ \varphi_0 = \left( b \mapsto a, c \mapsto a, d \mapsto 1, aba \mapsto c, aca \mapsto c, ada \mapsto b \right) \] and \[ \varphi_1 =  \varphi_0 \circ \mu_a = \left(b \mapsto c, c \mapsto d, d \mapsto b, aba \mapsto a, aca \mapsto a, ada \mapsto 1 \right)\] extend to homomorphisms from $H$ onto $G$. 

It is well known that the group $\mathfrak{G}$ is generated by four elements $a,b,c,d$ satisfying the relations defining $G$, so that we have a surjective map $\pi: G \rightarrow \mathfrak{G}$. The image of $H$ is the subgroup $\mathfrak{H}$ of $\mathfrak{G}$ stabilizing the first level of the rooted tree on which $\mathfrak{G}$ acts. Moreover there exist morphisms $\psi_0,\psi_1 : \mathfrak{H} \rightarrow \mathfrak{G}$ such that the diagrams 
$$
\begin{CD}
H @> \varphi_i >> G \\
@VV \pi V @VV \pi V\\
\mathfrak{H} @> \psi_i >> \mathfrak{G}\\
\end{CD}
$$
commute (\cite[\S 2]{Grig-minimal}). 

We shall check that the assumptions of Theorem \ref{thm-lim-fully} are satisfied. It is a simple verification that $\Phi = \left\{\varphi_0,\varphi_1\right\}$ satisfies (1), and (2) is clear since $\varphi_1 =  \varphi_0 \circ \mu_a$ by definition. Now to see that (3) is satisfied, by the commutativity of the above diagrams it is enough to see that the image of $\psi = (\psi_0,\psi_1) : \mathfrak{H} \rightarrow \mathfrak{G} \times \mathfrak{G}$ contains the diagonal embedding of $\mathfrak{G}$. This later fact follows for instance from Proposition 1 in \cite{Grig-minimal}. Therefore by applying Theorem \ref{thm-lim-fully}, we deduce that the group $G_{\Phi}$ is densely related. The statement follows, since the group $G_{\Phi}$ is precisely the group $\mathfrak{G}$ \cite[Th.\ 4]{Grig-Mama} (see also \cite{Grig-minimal}).
\end{proof}

Combined with Corollary \ref{cor-full-pres-cones}, Theorem \ref{thm-grigorchuk-dens-rel} implies the following result.

\begin{cor}
No asymptotic cone of the Grigorchuk group is simply connected.
\end{cor}

We shall point out that an explicit infinite presentation of the group $\mathfrak{G}$ has been given by Lysionok in \cite{Lys}, and this presentation has been shown to be minimal by Grigorchuk \cite{Grig-minimal}. However we stress out that in general the relation range cannot be seen on the lengths of relators even for a minimal presentation, see Remark \ref{rem-min-RR}. The presentation of the group $\mathfrak{G}$ from \cite{Lys} is an example of a finite endomorphic presentation in the sense of Bartholdi \cite{Bar-L-pres}. These are presentations which are obtained by successive iterations of (finitely many) endomorphisms of the free group to a finite initial subset of relators. However the proof of Theorem \ref{thm-lim-fully} does not seem to apply directly to every group with a finite endomorphic presentation (even in the case when the endomorphisms of the free group induce endomorphisms of the group associated to the endomorphic presentation). We do not know whether every such group is either finitely presented or densely related.

The group $\mathfrak{G}$ is an example of group that is generated by a finite state automaton. As a preliminary version of this article was circulating, the following conjecture has been communicated to us by L.\ Bartholdi. 

\begin{conj}[Bartholdi]
If a group $G$ is generated by a finite state automaton, then $G$ is either finitely presented or densely related.
\end{conj}

\subsubsection{Non-Hopfian groups}

A particular case of the construction carried out in \S \ref{subsec-non-hopf} is when $G$ is a non-Hopfian group, $H = G$ and $\Phi = \left\{\varphi\right\}$, where $\varphi: G \rightarrow G$ is a surjective, non-injective endomorphism of $G$.

\begin{cor} \label{cor-lim-hopf-fully}
Let $G$ be a compactly generated locally compact group, and assume that $\varphi: G \rightarrow G$ is a continuous, surjective, non-injective endomorphism of $G$ such that $K_{\varphi}$ is closed in $G$. Then the group $G_{\varphi} = G / K_{\varphi}$ is densely related.
\end{cor}

\begin{proof}
The assumption (1) is satisfied because $\varphi$ is non-injective. There is nothing to check for (2) in this case, and (3) also holds here because $\varphi$ is surjective. The conclusion therefore follows from Theorem \ref{thm-lim-fully}.
\end{proof}

Here are examples of groups that can be obtained as direct limits of non-Hopfian groups by iterating a non-injective endomorphism. It follows from Corollary \ref{cor-lim-hopf-fully} that all the groups appearing in the following examples are densely related.

\begin{ex} \label{ex-BSr}
Let $r\geq 1$, and let $m,n \geq 2$ be two coprime integers. Let us consider the group $\mathrm{C}(m,n,r)$ defined by the presentation \[ \mathrm{C}(m,n,r) = \left\langle x,t \, \mid \, t^r x^m t^{-r} = x^n, [txt^{-1},x] = \ldots = [t^{r-1}xt^{-(r-1)},x] =1 \right\rangle. \] This family of groups appeared in \cite{BaumStre} as a generalization of Baumslag-Solitar groups $\mathrm{BS}(m,n)$, which correspond to the case $r=1$. When $m,n \geq 2$ are chosen relatively prime, the endomorphism given by $\varphi(x)=x^m$ and $\varphi(t)=t$ is well defined, surjective but non-injective. In this case the limit group $\mathrm{C}(m,n,r)_{\varphi}$ is the metabelian group $\Z[1/mn]^r \rtimes_M \Z$, where the action is defined by the $r \times r$ companion matrix \[ M = \begin{pmatrix} 0 & \cdots & 0 & n/m \\ 1 & \ddots &  & 0 \\ & \ddots & \ddots  & \vdots \\ 0 & & 1 & 0 \end{pmatrix}.\]
\end{ex}

\bigskip

Let us now recall the construction from \cite{Meier} of non-Hopfian HNN-extensions. Let $A$ be a group and let $\mu, \nu: A \rightarrow A$ be two injective and non-surjective endomorphisms. Assume that
\begin{enumerate}[label=(\alph*)]
\item $\mu \circ \nu = \nu \circ \mu$;
\item $\left\langle \mu(A), \nu(A) \right\rangle = A$;
\item $\exists \, \alpha \notin \mu(A), \, \beta \notin \nu(A); \, [\nu(\alpha), \mu(\beta)] = 1$;
\end{enumerate}
and consider the HNN-extension of $A$ associated to the subgroups $\mu(A)$ and $\nu(A)$ \[ G = \left\langle A,t \, \mid \, t \mu(a) t^{-1} = \nu(a) \, \, \, \forall a \in A \right\rangle. \] 

The following was proved in \cite[Lem.\ 1]{Meier}. 

\begin{lem}
The endomorphism $\varphi: G \rightarrow G$ defined by $\varphi(t) = t$ and $\varphi(a) = \mu(a)$ for every $a \in A$, is surjective but non-injective.
\end{lem}

\begin{ex} \label{ex-LC-densely}
Let $\mathbf{K}_i$, $i=1,2$ be ultrametric local fields. Let $\mathcal{O}_i$ be the unique maximal compact subring in $\K_i$. Let $\pi_i$ be an element in $\K_i$ with $0<|\pi_i|<1$. Let us take $A = \mathcal{O}_1 \times \mathcal{O}_2$ and $\mu(x,y) = (x,\pi_2 y)$ and $\nu(x,y)=(\pi_1 x,y)$.

The HNN-extension $G$ defined above inherits a locally compact topology for which the inclusion of $A$ in $G$ is continuous and open. Since $\varphi(A) \subset A$ and $\varphi$ is injective on $A$, by Remark \ref{rem-K-closed} the subgroup $K_{\varphi}$ is a discrete subgroup of $G$. A simple computation shows that the group $G_{\varphi}$ is the group $(\mathbf{K}_1 \times \mathbf{K}_2) \rtimes \Z$, where the action is the multiplication by $(\pi_1,\pi_2^{-1})$.
\end{ex}


\begin{ex}
Discrete analogues of the construction of Example \ref{ex-LC-densely} may be carried out, for example with $A = \Z^2$ and $\mu(x,y) = (x,ny)$, $\nu(x,y)=(nx,y)$ for some integer $n \geq 2$. Here $G_{\varphi}$ is the metabelian group $\Z[1/n]^2 \rtimes \Z$, where the action is the multiplication by $(n,n^{-1})$.

Interesting examples also arise with $A$ non-abelian. Let for instance $A = H_3(\Z)$ be the Heisenberg group, and $\mu, \nu$  defined by \[ \mu: \begin{pmatrix} 1 & x & z \\ & 1 & y \\ & & 1 \end{pmatrix} \mapsto \begin{pmatrix} 1 & x & pz \\ & 1 & py \\ & & 1 \end{pmatrix}, \] \[ \nu: \begin{pmatrix} 1 & x & z \\ & 1 & y \\ & & 1 \end{pmatrix} \mapsto \begin{pmatrix} 1 & px & pz \\ & 1 & y \\ & & 1 \end{pmatrix}, \] where $p$ is a prime. Then one can check that the group $G_{\varphi}$ is the group \[ A_3(\Z[1/p]) = \begin{pmatrix} 1 & \Z[1/p] & \Z[1/p] \\ & p^{\Z} & \Z[1/p] \\ & & 1 \end{pmatrix}.\]
\end{ex}

\subsection{Finitely generated groups with a homomorphism to $\Z$}

\subsubsection{Preliminaries}

We say that an automorphism $\alpha: G \rightarrow G$ contracts into a subgroup $H$ if the sequence of subgroups $(\alpha^{-n}(H))_{n \geq 0}$ is increasing and ascends to $G$. If we have a homomorphism $\pi: G \twoheadrightarrow \Z$, we say that the action of $\Z$ contracts into a finitely generated subgroup of $G$ if there is $t \in G$ such that $\pi(t)$ is a generator of $\Z$ and the conjugation by $t$ contracts into a finitely generated subgroup of $G$. This is equivalent to saying that $\pi$ splits ascendingly over a finitely generated subgroup, i.e.\ that there is a decomposition of $G$ as an ascending HNN-extension whose associated homomorphism $G \twoheadrightarrow \Z$ is equal to $\pi$.

\bigskip


We now introduce some notation. Consider a finitely generated group $G=M\rtimes\Z$, and let $S=\{m_1,\dots,m_\ell\} \subset M$ and $t$ a generator for $\Z$ such that $S \cup \{t\}$ generates $G$. In the free group $F$ freely generated by the $m_i$ and $t$, we write $m_i^{[k]}=t^km_it^{-k}$. Let $R$ be the kernel of $F\to G$, and $R_n$ the normal subgroup of $F$ generated by the intersection of $R$ and the $n$-ball. 

For $u\le v$ let $M_{[u,v]}$ be the subgroup of $M$ generated by $\bigcup_{u\le k\le v}t^kSt^{-k}$. For every $n \geq 1$, let $G_n$ be the HNN-extension of $M_{[0,n]}$ along the isomorphism $M_{[0,n-1]}\to M_{[1,n]}$ given by conjugation by $t$. The inclusions $M_{[0,n]}\to M_{[0,n+1]}$ induce surjective homomorphisms $G_n\to G_{n+1}$. Let $R'_n$ be the kernel of the natural map $F \to G_n$.

\begin{lem}\label{rrp}
We have $R_{2n+3}\subset R'_n$.
\end{lem}

Note that typically, we have elements of the form $[m,t^{n+1}mt^{-n-1}]$ in $R_{2n+4}$ that are not in $R'_n$, so the lemma is optimal.

\begin{proof}
Consider an element $w$ in the $(2n+3)$-ball of $F$ that is a relation in $G$. The number of occurrences of $t$ and its inverse is even. First suppose this is $2n+2$: then up to cyclic conjugation, this element has the form $t^{2n+2}m$ and this is a contradiction. Hence it is at most $2n$. We can view the occurrences of $t$ as a loop in $\Z$ of length $\le 2n$; hence its diameter in $\Z$ is at most $n$. Therefore, after conjugating $w$ (and possibly increasing its length), we can suppose that this path in $\Z$ has range in $[0,n]$; this means that $w\in R'_n$.
\end{proof}


\subsubsection{Abelian-by-cyclic groups}

In this paragraph we prove that a finitely generated abelian-by-cyclic group that is not finitely presented must be densely related. Although a more general result will be obtained later, here we include an elementary proof for this class of groups.

\begin{lem} \label{lem-ij}
Assume that the action of $\Z$ does not contract into any finitely generated subgroup of $M$. Then for every $n \geq 1$ there exist $i,j$ such that $m_i \notin M_{[1,n]}$ and $m_j^{[n]} \notin M_{[0,n-1]}$.
\end{lem}

\begin{proof}
Argue by contradiction and assume that for all $i$, we have $m_i \in M_{[1,n]}$. This implies that $M_{[1,n]}=M_{[0,n]}$, and hence $G$ is an ascending HNN-extension over a finitely generated subgroup of $M$, a contradiction. 

Similarly, if $m_j^{[n]}\in M_{[0,n-1]}$ for all $j$, then $M_{[0,n-1]}=M_{[0,n]}$ and again $G$ is an ascending HNN-extension, resulting to a similar contradiction.
\end{proof}

\begin{prop} \label{prop-bracket}
Assume that the action of $\Z$ does not contract into any finitely generated subgroup of $M$. Then for every $n \geq 1$ there exist $i,j$ such that $[m_i,m_j^{[n+1]}]$ is not contained in $R'_n$. 
\end{prop}

\begin{proof}
Let $i,j$ be as in the conclusion of Lemma \ref{lem-ij}. Suppose that $[m_i,m_j^{[n+1]}]$ is a relation in $G_n$. We rewrite it as the equality in $G_n$
\[m_itm_j^{[n]}t^{-1}m_i^{-1}t(m_j^{[n]})^{-1}t^{-1}=1.\]
Once we view $m_i^{[n]}$ as an element of $M_{[0,n]}$, the Britton lemma implies that this expression is not reduced (with respect to the given HNN decomposition of $G_n$). This means that either $m_i\in M_{[1,n]}$, or $m_j^{[n]}\in M_{[0,n-1]}$, and both cases contradict the definition of $i$ and $j$.


Hence $[m_i,m_j^{[n+1]}]$ is not a relation in $G_n$, which exactly means that $[m_i,m_j^{[n+1]}]$ does not belong to $R'_n$. 
\end{proof}

\begin{cor} \label{cor-ab-par-Z}
Let $G = M \rtimes \Z$ be an abelian-by-cyclic finitely generated group. Then $G$ is either finitely presented or densely related.
\end{cor} 

\begin{proof}
We assume that $G$ is not finitely presented, and we show that $G$ is densely related. The action of $\Z$ cannot contract into a finitely generated subgroup of $M$, because otherwise $G$ would be an ascending HNN-extension over a finitely presented subgroup because $M$ is abelian. Therefore $G$ would be finitely presented, which is a contradiction. So we may apply Proposition \ref{prop-bracket}, which, combined with Lemma \ref{rrp}, says that the relation range of $G$ with respect to $S \cup \{t\}$ contains an element of the interval $[2n+4,4n+8]$ for every $n \geq 1$. This implies the statement.
\end{proof}

\subsubsection{Groups satisfying a law}

In this paragraph we deal with the larger class of groups $G$ satisfying a law. Recall that this means that there exists a non-trivial reduced word $w(x_1,\ldots,x_k)$ in some letters $x_1,\ldots,x_k$ such that $w(g_1,\ldots,g_k) = 1$ for all $g_1, \ldots, g_k \in G$. Note that if we have $G = M \rtimes \Z$, then $G$ satisfies a law if and only if $M$ satisfies a law.

We follow a similar approach as in the previous paragraph, except that $[m_i,m_j^{[n+1]}]$ need not be a relation anymore. We need the following simple lemma.

\begin{lem}\label{routine}
Consider a group amalgam $A\ast_C B$, and $a\in A\smallsetminus C$, $b\in B\smallsetminus C$. Then for every $k\in\Z\smallsetminus\{0\}$, we have $(ab)^k\notin A\cup B$.

If moreover we have, for all $k\in\Z\smallsetminus\{0\}$, $a^k\notin C$ and $b^k\notin C$, then $a,b$ freely generate a free subgroup.
\end{lem}

\begin{proof}
For $k\ge 1$ and $b_0\in B$, we can write $(ab)^kb_0^{-1}=abab\dots a(bb_0^{-1})$, which is a reduced word, hence it represents a non-trivial element by the Britton lemma. So $(ab)^k\notin B$, and the same proof shows that $(ab)^k\notin A$.

Under the stronger assumption, it first follows that $a,b$ are not torsion. Hence if they satisfy a non-trivial relation, it can be chosen of the form $a^{m_1}b^{n_1}\dots a^{m_r}b^{n_r}$ with $r\ge 1$ and $m_i,n_i$ nonzero integers. The stronger assumption again implies that this word is reduced.
\end{proof}

\begin{lem} \label{lem-free-sub}
Let $n \geq 1$ such that there exist $i,j$ so that $m_i \notin M_{[1,n]}$ and $m_j^{[n]} \notin M_{[0,n-1]}$. Then \[ a = m_i^{[-2]} m_j^{[n-1]} \, \, \text{and} \, \, b = m_i m_j^{[n+1]} \] freely generate a free subgroup in $G_n$.
\end{lem}

\begin{proof}
$G_n$ being a HNN-extension, it is a semidirect product $N\rtimes\Z$ with $N$ being a 2-sided iterated amalgam 
\[\dots M_{[-2,n-2]} *_{M_{[-1,n-2]}} M_{[-1,n-1]} *_{M_{[0,n-1]}} M_{[0,n]} *_{M_{[1,n]}} M_{[1,n+1]} *_{M_{[2,n+1]}}\dots,\]
which thus is the amalgam of two 1-sided iterated amalgams:
\[\left(\dots M_{[-2,n-2]} *_{M_{[-1,n-2]}} M_{[-1,n-1]}\right) *_{M_{[0,n-1]}} \left(M_{[0,n]} *_{M_{[1,n]}} M_{[1,n+1]} *_{M_{[2,n+1]}}\dots\right)\]
\[=: A\ast_C B.\]

Here $a = m_i^{[-2]} m_j^{[n-1]}$ belongs to $M_{[-2,n-2]} *_{M_{[-1,n-2]}} M_{[-1,n-1]}$. Note that by assumption $m_i^{[-2]} \in M_{[-2,n-2]} \smallsetminus M_{[-1,n-2]}$ and $m_j^{[n-1]} \in M_{[-1,n-1]} \smallsetminus M_{[-1,n-2]}$, so it follows from the first item in Lemma \ref{routine} that for all $k\in\Z\smallsetminus\{0\}$, we have $a^k\notin M_{[-2,n-2]} \cup M_{[-1,n-1]} = M_{[-2,n-1]}$.

The same argument applied to $b = m_i m_j^{[n+1]} \in M_{[0,n]} *_{M_{[1,n]}} M_{[1,n+1]}$ shows that $b^k \notin M_{[0,n+1]}$ for all $k\in\Z\smallsetminus\{0\}$. In particular neither $a^k$ nor $b^k$ belongs to $M_{[0,n-1]}$, and the second item in Lemma \ref{routine} implies that $a,b$ freely generate a free subgroup.
\end{proof}




The following is the main result of this paragraph.

\begin{prop} \label{prop-law-by-Z}
Consider a finitely generated group $G = M \rtimes \Z$ satisfying a law. Then either the action of $\Z$ contracts into a finitely generated subgroup of $M$, or the group $G$ is densely related.  
\end{prop}

\begin{proof}
Let $w(x_1,\ldots,x_k)$ be a reduced word in some letters $x_1,\ldots,x_k$ such that $w(g_1,\ldots,g_k) = 1$ for every $g_1, \ldots, g_k \in G$, and let $\ell$ be the length of $w$. 

Assume that the action of $\Z$ does not contract into any finitely generated subgroup of $M$. Lemmas \ref{lem-ij} and \ref{lem-free-sub} ensure the existence, for every $n \geq 1$, of elements $a_n, b_n \in F$, both of length $2n+4$, whose images in $G_n$ freely generate a free subgroup. For simplicity we still denote $a_n$ and $b_n$ the corresponding elements of $G_n$. Then one easily check that $a_nb_n, a_n^2 b_n^2, \ldots, a_n^k b_n^k$ freely generate a free subgroup of rank $k$ in $G_n$. In particular the word $w(a_nb_n,\ldots, a_n^k b_n^k)$ does not belong to $R_n'$, and by Lemma \ref{rrp} it does not belong to $R_{2n+3}$ either. But now by definition of $w$, the word $w(a_nb_n,\ldots, a_n^k b_n^k)$ belongs to $R$. It has length at most $2k \ell (2n+4)$, so the relation range of $G$ with respect to $S \cup \{t\}$ must contain an element in the interval $[2n+4, 2k \ell (2n+4)]$.
\end{proof}

We derive from Proposition \ref{prop-law-by-Z} the following consequence. Recall that a group is called coherent if every finitely generated subgroup is finitely presented. Examples of coherent groups are locally nilpotent groups, or locally finite groups.

\begin{cor} \label{cor-coherent+law}
Let $G = M \rtimes \Z$ be a finitely generated group, where $M$ is coherent. If $G$ satisfies a law, then $G$ is either finitely presented or densely related.
\end{cor}

\begin{proof}
Assume that $G$ is not densely related. Then by Proposition \ref{prop-law-by-Z} the action of $\Z$ contracts into a finitely generated subgroup of $M$, and $G$ is an ascending HNN-extension over a finitely presented group because $M$ is coherent. This clearly implies that $G$ is finitely presented. 
\end{proof}

\begin{cor} \label{cor-nilp-by-Z}
Every nilpotent-by-cyclic finitely generated group is either finitely presented or densely related.
\end{cor}

\begin{rem}
We shall point out that Corollary \ref{cor-nilp-by-Z} extends neither to nilpotent-by-abelian groups according to Proposition \ref{prop-RR-abels}, nor to (locally nilpotent)-by-cyclic groups in view of the examples from \cite[\S 3.5]{OOS}.
\end{rem}


\subsubsection{Metabelian groups}

The following important theorem, the proof of which is a combination of results of Bieri-Strebel \cite{BS78,BS80} and Bieri-Groves \cite{BG84}, characterizes infinitely presented groups among finitely generated metabelian groups. This result was explained and used in \cite{BCGS} to determine the finitely generated metabelian groups that are of extrinsic condensation (see \cite{BCGS} for this terminology).

\begin{thm}[Bieri-Groves-Strebel] \label{thm-metab-inf-res}
Let $G$ be a finitely generated metabelian group. Then $G$ admits a homomorphism onto $\Z$ that does not contract into any finitely generated subgroup if and only if $G$ is not finitely presented.\qed
\end{thm}


The following is the main result of this paragraph.

\begin{thm} \label{thm-metab-dichotomoie}
Let $G$ be a finitely generated metabelian group. Then $G$ is either finitely presented or densely related.
\end{thm}

\begin{proof}
If $G$ is not finitely presented, then Theorem \ref{thm-metab-inf-res} ensures the existence of a homomorphism from $G$ onto $\Z$ that does not contract into any finitely generated subgroup of $G$. Now since $G$ is metabelian, we may plainly apply Proposition \ref{prop-law-by-Z}, from which the conclusion follows. 
\end{proof}

The dichotomy appearing in Theorem \ref{thm-metab-dichotomoie} actually extends to the larger class of center-by-metabelian finitely generated groups.

\begin{cor} \label{cor-center-by-met}
Let $G$ be a center-by-metabelian finitely generated group. Then $G$ is either finitely presented or densely related.
\end{cor}

\begin{proof}
Let $Z$ be a central subgroup of $G$ such that $Q =G/Z$ is metabelian. If $Q$ is finitely presented then $Z$ has to be finitely generated by Lemma \ref{lem-gen-ext}, and therefore finitely presented since $Z$ is abelian. So $G$ lies in an extension of finitely presented groups, and is therefore finitely presented. 

Now assume that $Q$ is not finitely presented. Since $Q$ is metabelian, by Theorem \ref{thm-metab-inf-res} there is a homomorphism $Q \rightarrow \Z$ that does not contract into any finitely generated subgroup of $Q$. This homomorphism may clearly be lifted to a homomorphism $G \rightarrow \Z$ with the same property, and since $G$ is solvable, Proposition \ref{prop-law-by-Z} implies that $G$ is densely related.
\end{proof}

\nocite{*}
\bibliographystyle{amsalpha}
\bibliography{denselyrelated}

\end{document}